\documentclass[12pt,reqno]{amsart}
\usepackage{amssymb}
\usepackage{amsmath}
\usepackage{amsthm}
\usepackage{enumitem}
\usepackage{mathrsfs}
\usepackage[all]{xy}
\usepackage{hyperref}
\usepackage{marginnote}
\usepackage{bbm}
\usepackage{mathtools}

\usepackage[margin=1.5in]{geometry}

\RequirePackage{mathrsfs} 
\usepackage[OT2,T1]{fontenc}
\DeclareSymbolFont{cyrletters}{OT2}{wncyr}{m}{n}
\DeclareMathSymbol{\Sha}{\mathalpha}{cyrletters}{"58}

\newtheorem{theorem}{Theorem}[section]
\newtheorem{lemma}[theorem]{Lemma}
\newtheorem{proposition}[theorem]{Proposition}
\newtheorem{corollary}[theorem]{Corollary}
\newtheorem{conjecture}[theorem]{Conjecture}

\theoremstyle{definition}
\newtheorem*{ack}{Acknowledgements}

\newtheorem{remark}[theorem]{Remark}
\newtheorem{example}[theorem]{Example}
\newtheorem{definition}[theorem]{Definition}
\newtheorem{question}[theorem]{Question}

\numberwithin{equation}{section} \numberwithin{figure}{section}

\DeclareMathOperator{\Gal}{Gal}

\DeclareMathOperator{\Spec}{Spec}

 \DeclareMathOperator{\re}{Re}

\DeclareMathOperator{\Br}{Br} 

 \DeclareMathOperator{\res}{\partial}

 \DeclareMathOperator{\Norm}{N}
\DeclareMathOperator{\Frob}{Frob} \DeclareMathOperator{\dens}{dens}
\DeclareMathOperator{\Irr}{Irr} 
\DeclareMathOperator{\id}{id} \DeclareMathOperator{\Li}{Li}

\newcommand{\Adele}{\mathbf{A}}
\newcommand{\fp}{\mathfrak{p}}

\newcommand{\fa}{\mathfrak{a}}

\newcommand{\br}{\mathscr{B}}

\newcommand\FF{\mathbb{F}}
\newcommand\PP{\mathbb{P}}

\newcommand\ZZ{\mathbb{Z}}
\newcommand\NN{\mathbb{N}}
\newcommand\QQ{\mathbb{Q}}
\newcommand\RR{\mathbb{R}}
\newcommand\CC{\mathbb{C}}
\newcommand\GG{\mathbb{G}}
\newcommand\Ga{\GG_\mathrm{a}}

\newcommand\OO{\mathcal{O}}

\newcommand{\x}{\mathbf{x}}

\newcommand{\one}{\mathbbm{1}}

\title[Fibrations with few rational points]{Fibrations with few rational points}

\author{D. Loughran}
\address{D. Loughran \\
University of Manchester,
School of Mathematics,
Oxford Road,
Manchester,
M13 9PL,
UK.}
\email{loughran@math.uni-hannover.de}

\author{A. Smeets}
\address{A. Smeets \\ Radboud University Nijmegen, IMAPP, Heyendaalseweg 135, 6525 AJ Nijmegen, The Netherlands   \emph{and} KU Leuven, Departement Wiskunde, Celestijnenlaan 200B, 3001 Leuven, Belgium.} 
\email{arnesmeets@gmail.com}


\begin{document}

\begin{abstract}
	We study the problem of counting the number of varieties
	in families which have a rational point.
	We give conditions on the singular fibres that force
	very few of the varieties in the family to contain a rational point,
	in a precise quantitative sense. This generalises and unifies existing results in the literature
	by Serre, Browning--Dietmann, Bright--Browning--Loughran, Graber--Harris--Mazur--Starr, \emph{et al}.
\end{abstract}

\maketitle

\thispagestyle{empty}

\tableofcontents

\section{Introduction}
Given a family of varieties over a number field $k$, a natural question is the following: how ``many'' varieties in the family contain a rational point? To make this question more precise, we shall use height functions. 

Let $X$ be a variety over $k$ equipped with a dominant morphism $\pi:X \to \PP_k^n$. Consider $\PP_k^n$ with its usual height: given $x = (x_0:\cdots:x_n) \in \PP^n(k)$, we define  $$H(x) = \prod_{v } \max\{|x_0|_v, \ldots, |x_n|_v\}$$ where the product is over all places $v$ of $k$ and $|\cdot|_v$ denotes the normalised $v$-adic absolute value.
One is then interested in studying the quantity
\begin{equation} \label{def:rational}
N(\pi,B)=\#\{x \in \PP^n(k): x \in \pi(X(k)), H(x) \leq B\},
\end{equation}
as $B \to \infty$, which counts those varieties in the family with a rational point. In this paper, we focus  on the function
\begin{equation} \label{def:local}
N_{\mathrm{loc}}(\pi,B) = \#\{x \in \PP^n(k): x \in \pi(X(\Adele_k)), H(x) \leq B\},
\end{equation}
counting those varieties in the family which are everywhere locally solvable. The goal of this paper is to obtain upper bounds for the latter quantity, the primary motivation being to give upper bounds for the counting function \eqref{def:rational}. This will allow us to deduce that $0\%$ of the varieties in certain families admit a rational point, by proving the stronger statement that $0\%$ of the varieties are everywhere locally solvable. The counting functions (\ref{def:rational}) and (\ref{def:local}) have been studied for various special families by numerous authors, see e.g.~\cite{BD09}, \cite{Guo95}, \cite{Hoo93}, \cite{Hoo07}, \cite{PV04}, \cite{Ser90} and the more recent papers \cite{Bha14}, \cite{BGW15}, \cite{BBL15}, \cite{BD14} and \cite{Lou13}.

Let us first state a special case of our results, which illustrates the kind of behaviour observed in this paper:

\begin{theorem} \label{thm:single_fibre}
	Let $X$ be a  proper, smooth irreducible algebraic variety over a number field $k$, equipped with a dominant morphism $\pi:X \to \PP_k^n$ with geometrically integral generic fibre. If the fibre over some codimension one point of $\PP_k^n$ is irreducible, but not geometrically integral, then
	$$N_{\mathrm{loc}}(\pi,B) = o(B^{n+1}).$$
\end{theorem}
Thus  a simple geometric condition implies that ``almost all'' varieties in the family are not everywhere locally solvable; indeed, recall that by \cite{Sch79}, there is some $c_k > 0$ such that $\#\{x \in \PP^n(k): H(x) \leq B\} \sim c_kB^{n + 1}$ as $B\to \infty$.

To state our main results, we need more notation. Following Skorobogatov \cite[Def.~0.1]{Sko96}, a scheme over a field $k$ is said to be \emph{split} if it contains a geometrically integral open subscheme, and \emph{non-split} otherwise. 
Let $\pi: X \to \PP_k^n$ be as above and choose some model $\pi:\mathcal{X} \to \PP_k^n$ for $\pi$ over $\OO_{k,S}$
for some finite set of primes $S$ of $k$ (denoted again by $\pi$). We define
\begin{equation} \label{def:Delta}
\Delta(\pi)
 := \lim_{B \to \infty} \frac{\sum_{\substack{\fp \subset \OO_{k,S} \\ N(\fp) \leq B}} \#\{x_\fp \in \PP^n(\FF_\fp) :\pi^{-1}(x_\fp) \mbox{ is non-split}\}}
{\sum_{\substack{\fp \subset \OO_{k,S} \\ N(\fp) \leq B}}N(\fp)^{n-1} },
\end{equation}
where the sums are taken over the non-zero prime ideals $\fp$ of $\OO_{k,S}$.
We will show that this limit exists, and give a formula to calculate it in terms of the splitting behaviour
of the fibres of $\pi$ over codimension one points. This will reduce the calculation of $\Delta(\pi)$ to a problem in group theory.

To describe this formula, let $D$ be a codimension one point of $\PP_k^n$ with residue field $\kappa(D)$. The fibre of $\pi$ over $D$ is a possibly reducible $\kappa(D)$-scheme. 
Let $I_D(\pi)$ be its set of geometric irreducible components of multiplicity one, i.e.~the irreducible components of $\pi^{-1}(D) \otimes_{\kappa(D)} \overline{\kappa(D)}$ which are generically reduced.
Choose a finite group $\Gamma_D(\pi)$ through which the action of $\Gal(\overline{\kappa(D)}/\kappa(D))$ on $I_D(\pi)$ factors; this corresponds to a choice of splitting field for these irreducible components, cf. \S\ref{sec:delta} for more details. If $I_D(\pi)$ is non-empty, we set
\begin{equation} \label{def:delta_D}
\delta_D(\pi) := \frac{\#\{\gamma \in \Gamma_D(\pi) : \gamma \mbox{ acts with a fixed point on } I_{D}(\pi)\}}{\# \Gamma_D(\pi)}.
\end{equation}
This definition is independent of the choice of $\Gamma_D(\pi)$. If $I_D(\pi)$ is empty, we set $\delta_D(\pi) = 0$. If $\pi^{-1}(D)$ is split, then it is simple to see that $\delta_D(\pi) = 1$, since then one element of $I_D(\pi)$ is fixed by all elements of $\Gamma_D(\pi)$. However the converse does not hold in general; see Example \ref{Ex:CT} for a counterexample.

We now come to the main result of this paper:

\begin{theorem} \label{thm:non-split}
	Let $X$ be a proper, smooth algebraic variety over a number field $k$, equipped with a dominant morphism $\pi:X \to \PP_k^n$ with geometrically integral generic fibre. 
	Then the limit \eqref{def:Delta} exists and we have
	$$\Delta(\pi) = \sum_{D \in (\PP_k^n)^{(1)}}\left(1 - \delta_D(\pi)\right).$$
	Moreover, we have the upper bound
	$$N_{\mathrm{loc}}(\pi,B) \ll  \frac{B^{n+1}}{(\log B)^{\Delta(\pi)}}.$$
\end{theorem}

This result illustrates the philosophy developed in  \cite[\S1.3]{Lou13}, namely that the behaviour of $N_{\mathrm{loc}}(\pi,B)$ should be determined by the non-split fibres over codimension one points. Note that no assumptions are made on the smooth fibres. Theorem \ref{thm:non-split} allows us to recover and extend numerous existing results. For example, we generalise results of Serre \cite{Ser90} on conic bundles, and improve upon work of Browning--Dietmann \cite{BD09} on Fermat curves and work of Graber--Harris--Mazur--Starr \cite{GHMS04} on genus $1$ fibrations, cf.~\S\ref{Sec:examples}.

If $\Delta(\pi)>0$, then Theorem \ref{thm:non-split} implies that $0\%$
of the varieties in the family are everywhere locally solvable. Our next result shows
that the converse is also true. Hence we obtain a complete description of when 
a positive proportion of the varieties in a family are everywhere locally solvable.

\begin{theorem} \label{thm:iff}
	Let $X$ be a proper, smooth algebraic variety over a number field $k$, equipped
	with a dominant morphism $\pi:X \to \PP_k^n$ with geometrically integral generic fibre. 
	Assume that 
	$X(\Adele_k) \neq \emptyset$ and $\Delta(\pi)=0$.
	Then 
	$$\lim_{B \to \infty} \frac{N_{\mathrm{loc}}(\pi,B)}{\#\{x \in \PP^n(k): H(x) \leq B\}}$$
	exists, is non-zero and is a product of local densities.
\end{theorem}

The proof of this result is an adaptation of the proof of \cite[Thm.~1.3]{BBL15}, and uses
the sieve of Ekedahl \cite{Eke91}.
Examples of families for which the product of local densities has been explicitly
computed can be found in \cite{BCFJK15}, \cite{BCF16} and \cite{BBL15}. 

Theorem \ref{thm:iff} yields the following corollary:

\begin{corollary} \label{cor:Hasse}
	Let $X$ be a proper, smooth algebraic variety over a number field $k$, equipped
	with a dominant morphism $\pi:X \to \PP_k^n$ with geometrically integral generic fibre. Assume that $\Delta(\pi)=0$ and that,
	outside of a thin subset of $\PP^n(k)$, the smooth fibres of $\pi$ over rational points satisfy the Hasse principle. 
	
	Then $X$ satisfies the 
	Hasse principle.
\end{corollary}
We use the term ``thin set'' in the sense of Serre \cite[\S9.1]{Ser97b}. The result follows immediately from  Theorem \ref{thm:iff}, as thin sets have density zero \cite[Thm.~13.1.3]{Ser97b}.

There is a large industry of proving results like Corollary \ref{cor:Hasse} via
the fibration method (see e.g. \cite{HW15} for a recent highlight).
One usually requires that very few fibres over codimension
one points are non-split, or that non-split fibres only occur over linear subspaces.
In Corollary \ref{cor:Hasse} however, we allow arbitrarily many non-split 
fibres over arbitrary codimension one points $D$, provided that each such $D$ satisfies $\delta_D(\pi) = 1$. It would be interesting to construct new examples of fibrations with non-split fibres over codimension $1$ points and $\Delta(\pi) = 0$, beyond those already studied by Colliot-Th\'el\`ene in \cite{CT14}.

Our last result is a version of Theorem \ref{thm:non-split}
for integral points. Let us write $k_\infty = \OO_k \otimes_\ZZ \RR$, and let $\| \cdot \|$ be an $\RR$-vector space norm
on $k_\infty^{n}$. We identify $\OO_k^n$ with its image in $k_\infty^n$, viewed as a full rank sublattice.
Given $\pi:X \to \mathbb{A}_k^n$, define
\begin{equation} \label{def:Nloc_affine}
	N_{\mathrm{loc}}(\pi,B) = \#\{\x \in \OO_k^n: \x \in \pi(X(\Adele_k)), \|\x\| \leq B\}.
\end{equation}

\begin{theorem} \label{thm:non-split_integral}
	Let $X$ be a smooth algebraic variety over a number field $k$, equipped
	with a proper, dominant morphism $\pi:X \to \mathbb{A}_k^n$ with geometrically integral generic fibre. Then
	$$N_{\mathrm{loc}}(\pi,B) \ll  \frac{B^{n}}{(\log B)^{\Delta(\pi)}},
	\quad \mbox{where } \quad
	\Delta(\pi) = \sum_{D \in (\mathbb{A}_k^n)^{(1)}}\left(1 - \delta_D(\pi)\right).$$
\end{theorem}
Here $\delta_D(\pi)$ is defined in a manner similar to $\eqref{def:delta_D}$.
The analogue here of Theorem \ref{thm:iff} does not hold in general, due to possible obstructions at the real places (for example if $\pi_v(X(k_v))$ is bounded for some real place $v$).

The proofs of our results require input from geometry and analytic number theory. The key geometric
ingredient (building on ideas of Wittenberg) is Theorem \ref{thm:sparsity}. Given  $\pi$ as above and a finite place $v$ of $k$, this theorem gives sufficient
conditions under which the fibre of $\pi$ above a closed point, which is $v$-adically close to a closed point with non-split fibre, has no $k_v$-point; one should see this as certain local points being ``sparse'' close to a non-split fibre.
An easy example illustrating this type of behaviour (used by Serre in \cite{Ser90}) is the following: let $a \in k^*$ be a non-square and let $v$ be a finite place not dividing $2$ such that $a \not \in k_v^{* 2}$. Consider the family of conics
\begin{equation} \label{serre}
C: \quad ax^2 + ty^2 = z^2  
\end{equation}
over $\mathbb{A}^1_k$. If $t$ has valuation $1$ with respect to $v$, then $C_t$ has no $k_v$-point. It is precisely this type of observation which we will generalise in \S \ref{Sec:sparsity}.

As for the analytic input, the criterion provided by Theorem \ref{thm:sparsity} is ideal for an application of the large sieve. In order to perform the sieve, we will use versions of the Chebotarev density theorem for arithmetic schemes  proven by Serre \cite[\S9]{SerNXp}.
This builds on the method of Serre for conic bundles \cite{Ser90}.

We end with the following conjecture:
\begin{conjecture} \label{question}
	Let $\pi:X\to \PP_k^n$ be as in Theorem \ref{thm:non-split}. Assume that at least one fibre of $\pi$ is everywhere locally solvable
	and that the fibre of $\pi$ over every codimension one
	point of $\PP_k^n$ has an irreducible component of multiplicity one. 
	
	Then the bounds given in Theorem \ref{thm:non-split} are sharp.
\end{conjecture}

This conjecture is known to hold in some special cases, 
e.g.~for some families of conics \cite{Guo95}, \cite{Hoo93}, \cite{Hoo07}, \cite{Sof16}, Severi--Brauer varieties \cite{Lou13} and
norm one tori \cite{BN15} (see \S \ref{sec:tori}). Moreover, when $\Delta(\pi) = 0$, Theorem \ref{thm:iff} proves that Conjecture \ref{question} holds; this result supersedes various special cases, e.g.~\cite[Thm.~3.6]{PV04}. This conjecture builds on the conjectural framework of the first-named author began in \cite[\S1.3]{Lou13}. Note that Theorem \ref{thm:non-split} will certainly not be sharp as soon as there are many codimension one points $D$ for which $\pi^{-1}(D)$ contains no irreducible component of multiplicity one. In such cases far fewer fibres tend to be everywhere locally solvable (see \cite{CTSSD97} for examples of this phenomenon). Sieves do not yield good upper bounds in such cases; for example, the large sieve gives very poor upper bounds for counting squareful integers.

We end this introduction with an overview of the paper. In \S\ref{Sec:sparsity}, we prove the sparsity criterion for local points needed for the sieving process. Some basic properties of the splitting densities appearing in our results are proven in \S\ref{Sec:split}. The main results will be proven in \S\ref{Sec:proofs}. Finally, in \S\ref{Sec:examples}, we give examples showing that our results generalise and improve upon known special cases in the literature, together with new applications. 

\subsection*{Notation} If $Y$ is a noetherian scheme, the set of codimension $i$ points will be denoted by $Y^{(i)}$, and the set of closed points by $\underline{Y}$. For $y \in Y$, we denote the residue field of $y$ by $\kappa(y)$.
A variety over a field $k$ is a reduced, separated scheme of finite type over $k$. If $k$ is a number field, by a prime (or finite place) $\fp$ of $k$ we mean a prime ideal of the ring of integers $\OO_k$. The completion of $k$ at $\fp$ will be denoted by $k_\fp$, with local ring $\OO_\fp$ and residue field $\FF_\fp$. By $\mathbf{A}_k$, we mean the ring of ad\`eles (whereas $\mathbb{A}^n_k$ will denote affine $n$-space over $k$). 

Given a scheme $X$, the multiplicity of an irreducible component $Z$ of $X$ is defined to be the length
of the local ring of $X$ at the generic point of $Z$ (see \cite[\S 1.5]{fulton}). 
In particular, an \emph{irreducible component of multiplicity $1$} is one which is generically reduced.

\begin{ack}
	We benefited from helpful discussions with Fran\c{c}ois Brunault, Jean-Louis Colliot-Th\'{e}l\`{e}ne, Matthieu Romagny, 
	Matthias Sch\"{u}tt, Alexei Skorobogatov, Efthymios Sofos and Michael Stoll. We thank Olivier Wittenberg for crucial input on the notion of submersivity used in \S\ref{Sec:sparsity} and for discussions on elliptic fibrations. We are grateful to the referee for a careful reading of our paper and many useful comments.
	The second-named author was supported by ERC grant MOTMELSUM (R. Cluckers) and is a postdoctoral fellow of FWO Vlaanderen (Research Foundation -- Flanders).

\end{ack}

\section{Sparsity of local points around non-split fibres} \label{Sec:sparsity}

We gather some facts on the notion of \emph{submersivity} of a morphism of schemes in \S\ref{sec:3.1}. This paragraph is based on unpublished notes of Olivier Wittenberg, who generously shared his thoughts with us. In \S\ref{sec:3.2}, we prove the crucial ``sparsity criterion'' for local points around non-split fibres, needed in \S\ref{Sec:proofs}.

\subsection{Submersive morphisms and regularity} \label{sec:3.1}

\begin{definition}
A morphism of schemes $f: X \to Y$ is \emph{submersive at a point $x \in X$} if the natural map $T_xX \to T_yY \otimes_{\kappa(y)} \kappa(x)$ is surjective, where $y = f(x)$; here $T_xX$ denotes the tangent space of $X$ at $x$.
We say that $f$ is \emph{submersive} if it is submersive at every point of $X$. \end{definition}

The following lemma is trivial.

\begin{lemma} \label{lem:composition} Let $f: X \to Y$ and $g: Y \to Z$ be morphisms of schemes. If $f$ is submersive at $x \in X$ and $g$ is submersive at $f(x)$, then $g \circ f$ is submersive at $x$. In particular, if $f$ and $g$ are submersive, then $g \circ f$ is submersive.

Conversely, if $g \circ f$ is submersive at $x \in X$, then $g$ is submersive at $f(x)$. In particular, if $g \circ f$ is submersive and $f$ is surjective, then $g$ is submersive.
\end{lemma}

The following result is well-known:

\begin{proposition} Smooth morphisms are submersive. \end{proposition}

Indeed, any smooth morphism $f: X \to Y$ factors locally as the composition of an \'etale morphism $h: X \to \mathbb{A}^n_Y$ (for some $n$) and the projection map $g: \mathbb{A}^n_Y \to Y$. It is not hard to check that both $g$ and $h$ are submersive. The result then follows from the previous lemma.



\begin{proposition}\label{submersivity} Let $Y$ be an integral scheme with generic point $\eta$ such that $\mathrm{char}\,\kappa(\eta) = 0$. Let $f: X \to Y$ be a morphism of finite type. Then there exists a dense open subscheme $U \subseteq Y$ such that $f_U: X_U \to U$ is submersive. \end{proposition}

The hypothesis on the characteristic is necessary for such a statement to be true; consider for example $Y = \Spec(k[t,\frac{1}{t}])$ and $X = \Spec(k[t,\frac{1}{t},x]/(x^p - t))$, where $k$ is an algebraically closed field of characteristic $p > 0$. Then $X \to Y$ is submersive at the unique point of the generic fibre, but nowhere else.

\begin{proof} By Lemma \ref{lem:composition}, we can assume that $X$ is reduced and, shrinking $Y$ if necessary, that $X$ is flat over $Y$. We will induct on the dimension of the generic fibre $X_\eta$ of $f$. As $f$ is flat and finitely presented, it is open. Therefore if $X_\eta = \emptyset$ then $X = \emptyset$; thus in this case there is nothing to prove. 

So assume that $X_\eta \neq \emptyset$. Let $V \subseteq X$ be the smooth locus of $f$ and denote by $W$ its complement, seen as a reduced closed subscheme of $X$. Since $\kappa(\eta)$ is a field of characteristic zero and $X_\eta$ is reduced, we have $\dim W_\eta < \dim X_\eta$. We can therefore assume, using the induction hypothesis, and shrinking $Y$ if necessary, that the composition $W \hookrightarrow X \to Y$ is submersive. The second half of Lemma \ref{lem:composition} then implies that $f$ is submersive at the points in $W$; hence $f$ is submersive globally, since it is even smooth at the points in $V$.
\end{proof}

Given a regular scheme $X$ and two regular, closed subschemes $Z_1$ and $Z_2$ of $X$, we say that $Z_1$ and $Z_2$ intersect \emph{transversally} at $x$ if $T_x Z_1 + T_x Z_2 = T_x X$ as $\kappa(x)$-subspaces of the vector space $T_x X$.

\begin{proposition} \label{regularity}
Let $f: X \to Y$ be a flat morphism of finite type between regular schemes. Let $Z_1$ and $Z_2$ be regular closed subschemes of $Y$ meeting transversally. Assume that $f$ is smooth above $Y \setminus Z_1$ and that the induced morphism $f: X \times_Y Z_1 \to Z_1$ is submersive. Then $X \times_Y Z_2$ is regular. 
\end{proposition}
\begin{proof} It suffices to prove that $X_{Z_2} := X \times_Y Z_2$ is regular above the points in $Z_1 \cap Z_2$. Let $x \in X_{Z_2}$ be such a point, i.e. $y=f(x) \in Z_1 \cap Z_2$. We claim that there is a short exact sequence of $\kappa(x)$-vector spaces \begin{equation} 0 \to T_xX_{Z_2} \to T_x X \to (T_{y} Y/T_{y} Z_2) \otimes_{\kappa(y)} \kappa(x) \to 0. \label{SES} \end{equation} Only right exactness of the above sequence is non-trivial. This follows from the following simple observations: the image of $T_{y} Z_1$ generates $T_{y} Y/T_{y} Z_2$, by the transversality assumption; and the map $T_x {X_{Z_1}} \to T_{y}Z_1 \otimes_{\kappa(y)} \kappa(x)$ is surjective, since $f: X_{Z_1} \to Z_1$ is submersive. Next, the exactness of (\ref{SES}) implies that $$\dim_{\kappa(x)} T_x X_{Z_2} = \dim \mathcal{O}_{X,x} - \dim \mathcal{O}_{Y,y} + \dim \mathcal{O}_{Z_2,y}$$ since $X$, $Y$ and $Z_2$ are regular. Since $X$ is flat over $Y$, \cite[Thm. 14.2.1]{EGAIV} gives $$\dim \mathcal{O}_{X,x} = \dim \mathcal{O}_{Y,y} + \dim \mathcal{O}_{X_{y},x}$$ where $X_y$ is the fibre of $f$ above $y$. Hence (again by  \cite[Thm. 14.2.1]{EGAIV}) $$\dim_{\kappa(x)} T_x X_{Z_2} = \dim \mathcal{O}_{X_{y},x} + \dim \mathcal{O}_{Z_2,y} = \dim \mathcal{O}_{X_{Z_2},x}.$$ It follows that $X_{Z_2}$ is regular at $x$, as required.\end{proof}

\subsection{Sparsity of local points} \label{sec:3.2}

We start with a trivial observation:

\begin{lemma}\label{lem:smooth_point}
	Let $X$ be a scheme of finite type over a field $k$. If $X$ admits a smooth rational point, then $X$ is split.
\end{lemma}
\begin{proof}
	Any smooth point admits a smooth, connected open neighbourhood. If the smooth point is a \emph{rational} point, then this neighbourhood is in fact geometrically connected, and hence geometrically integral.
\end{proof}

We will also need the following elementary geometric fact.

\begin{lemma} \label{flatness}
 Let $f: X \to Y$ be a dominant morphism of finite type between integral noetherian schemes. If $Y$ is regular in codimension $1$, there exists a closed subscheme $F \subseteq Y$, of codimension at least $2$ in $Y$, such that if $U = Y \setminus F$, then the induced morphism $f^{-1}(U) \to U$ is flat. \end{lemma}

\begin{proof} Without the condition on the codimension of $F$, this is well-known (``generic flatness''). The fact that one can take $F$ to be of codimension at least $2$ if $Y$ is regular in codimension $1$ follows from \cite[Prop.~3.9.7]{hartshorne}. \end{proof}

Let us now state the main result of this section, which will be fundamental for the sieving process carried out in \S\ref{Sec:proofs}:

\begin{theorem} \label{thm:sparsity} 

Let $k$ be a number field. Let $f: X \to Y$ be a proper, dominant morphism of smooth, geometrically integral $k$-varieties. Choose a finite set of finite places $S$ for which there exist regular $\mathcal{O}_{k,S}$-models $\mathcal{X}$ and $\mathcal{Y}$ for $X$ and $Y$ respectively, and such that $f$ extends to a proper morphism $\mathcal{X} \to \mathcal{Y}$. Let $\mathcal{T}$ be any reduced divisor on $\mathcal{Y}$ which contains the locus where this morphism is not smooth. 

Enlarging $S$ if needed, one can find a closed subset $\mathcal{Z}$ of $\mathcal{T}$ containing the singular locus of $\mathcal{T}$ and of codimension $2$ in $\mathcal{Y}$, satisfying the following property.

Let $\mathfrak{p} \not\in S$ be a finite place of $k$ and choose $\mathcal{P} \in \mathcal{Y}(\mathcal{O}_\mathfrak{p})$ such that the image of the morphism $\mathcal{P}: \mathrm{Spec}\,\mathcal{O}_\mathfrak{p} \to \mathcal{Y}$ meets $\mathcal{T}$ transversally outside of $\mathcal{Z}$ and such that the fibre above $\mathcal{P}\ \mathrm{mod}\ \mathfrak{p} \in \mathcal{T}(\FF_\mathfrak{p})$ is non-split. If $P \in Y(k_\mathfrak{p})$ denotes the generic fibre of $\mathcal{P}$, then $f^{-1}(P)$ does not have any $k_\mathfrak{p}$-points.
 
\end{theorem}

As an example, take $Y = \mathbb{P}^1_k$ and let $P \in \mathbb{P}^1_k(k_\mathfrak{p}) \setminus \{\infty\}$ for a finite place $\mathfrak{p}$ which is sufficiently large, with coordinates $(t_P : 1)$ for some $t_P \in \mathcal{O}_\mathfrak{p}$. Assume  that the fibre above  $\mathcal{P}\ \mathrm{mod}\ \mathfrak{p} \in \mathbb{P}^1_{\mathcal{O}_k}(\FF_\mathfrak{p})$ is non-split. If $t_Q \in \mathcal{O}_\fp$ satisfies $v_\mathfrak{p}(t_P - t_Q) = 1$ then the fibre above the point $Q \in \mathbb{P}^1_k(k_\mathfrak{p})$ with coordinates $(t_Q : 1)$ does not have a $k_\mathfrak{p}$-point. In particular, in the special case of conic bundles over $\mathbb{P}^1_k$ given by (\ref{serre}), we recover Serre's observation used in \cite{Ser90}.

We now prove the theorem:

\begin{proof} Let $\mathcal{T}_1,\mathcal{T}_2,\ldots,\mathcal{T}_r$ be the irreducible components of $\mathcal{T}$. Enlarging $S$ if necessary, we can assume that each $T_i = \mathcal{T}_i \otimes_{\OO_{k,S}} k$ is non-empty and hence a divisor on $Y$. Then by Lemma \ref{flatness}, each $T_i$ contains a strict closed subset $F_i$ such that if $F = \bigcup_{i = 1}^r F_i$, then $f$ is flat over $Y \setminus F$. Denoting by $\mathcal{F}$ (resp.~$\mathcal{F}_i$) the closure of $F$ (resp.~$F_i$) in $\mathcal{T}$ (resp.~$\mathcal{T}_i$) and enlarging $S$ if necessary, we may assume that $\mathcal{X} \to \mathcal{Y}$ is flat away from the codimension $2$ subset $\mathcal{F}$.

For each $i = 1,2,\ldots, r$, there exists a strict closed subset $\mathcal{S}_i \subseteq \mathcal{T}_i$ such that $f|_{\mathcal{T}_i}$ is submersive over $\mathcal{T}_i \setminus \mathcal{S}_i$, by Proposition \ref{submersivity}. Enlarging $\mathcal{S}_i$ if necessary, we may assume that $\mathcal{T}_i \setminus \mathcal{S}_i$ is regular. Enlarging $S$ if necessary, we may now assume that the induced morphism $\mathcal{X} \times_{\mathcal{Y}} (\mathcal{T} \setminus \mathcal{S}) \to \mathcal{T} \setminus \mathcal{S}$ is submersive, where $\mathcal{S} := \bigcup_{i = 1}^r \mathcal{S}_i$ is a closed subset of codimension at least two in $\mathcal{Y}$.

We now take $\mathcal{Z} = \mathcal{F} \cup \mathcal{S}$ to be the required subset of codimension $2$ in $\mathcal{Y}$, and choose $\mathcal{P} \in \mathcal{Y}(\mathcal{O}_\mathfrak{p})$ as in the statement of the theorem. Proposition \ref{regularity} implies that the fibre product $\mathcal{X} \times_{\mathcal{Y}} \mathcal{P}$ is a regular $\mathcal{O}_\mathfrak{p}$-scheme. If its generic fibre $f^{-1}(P)$ were to have a $k_\mathfrak{p}$-point, then $\mathcal{X} \times_{\mathcal{Y}} \mathcal{P}$ would have an $\mathcal{O}_\mathfrak{p}$-point by properness, which then has to specialise to a rational point in the smooth locus of its special fibre; this observation crucially uses the regularity of $\mathcal{X} \times_{\mathcal{Y}} \mathcal{P}$, cf.~\cite[Prop.~3.1.2]{BLR90}. By Lemma \ref{lem:smooth_point}, this would imply that the fibre above $\mathcal{P}\ \mathrm{mod}\ \mathfrak{p}$ is split, contradicting our assumptions. This finishes the proof. \end{proof}
\section{Splitting densities} \label{Sec:split}

The aim of this section is to study the fundamental properties of the quantity $\Delta(\pi)$ defined in \eqref{def:Delta}. We will show that it is a well-defined birational invariant of the generic fibre (in a sense we will make precise in Lemma \ref{lem:delta_birational}) and give a closed formula which calculates its value (Proposition \ref{prop:Delta}). 
\subsection{Frobenian sets}
We first recall the notion of frobenian sets
on arithmetic schemes, following Serre's treatment in \cite[\S9.3]{SerNXp}.

Let $\mathcal{Y}$ be a flat integral scheme of finite type over $\ZZ$ of dimension $n + 1$, so that the generic fibre $Y = \mathcal{Y} \otimes_{\mathbb{Z}} \mathbb{Q}$ has dimension $n$. Recall that we denote
the set of closed points of $\mathcal{Y}$ by $\underline{\mathcal{Y}}$. The \emph{norm} of closed point $y \in \underline{\mathcal{Y}}$ is defined
to be $\Norm(y) = \#\kappa(y)$. The \emph{degree} $\deg(y)$ of a closed point $y$ is the degree
of its residue field $\kappa(y)$ over its prime subfield. Let $K$ be the function field of $\mathcal{Y}$. 

\begin{definition} \label{def:frob}
A \emph{frobenian set} is a subset $F \subset \underline{\mathcal{Y}}$ with the following properties. There exist a non-Zariski-dense subset $S \subset \underline{\mathcal{Y}}$, a finite Galois extension $L/K$ with Galois group $\Gamma$ and a subset $C \subset \Gamma$ such that:
\begin{enumerate}
	\item $C$ is stable under conjugation;
	\label{item:conjugation}
	\item the normalisation of $\mathcal{Y}$ in $L$ is finite 
	\'{e}tale over the points in $\underline{\mathcal{Y}} \setminus S$;
	\label{item:etale} 
	\item for all $y \in \underline{\mathcal{Y}} \setminus S$, we have $y \in F$ if and only if $\Frob_y \in C$.
	\label{item:Frob}
\end{enumerate}
Here $\Frob_y \in \Gamma$ denotes the Frobenius element at $y \in \underline{\mathcal{Y}}$, cf.~\cite[9.3.1.3]{SerNXp}.
We define the \emph{mean} of $F$ to be $m(F) = |C|/|\Gamma|$ (this is easily seen to be independent of the choice of $\Gamma$ and $C$).
\end{definition}

\begin{example} Let $E/K$ be a finite extension of number fields. Then the set of all prime ideals of $\OO_K$ which split completely in $E$ is a frobenian set. 

To see this, in the notation of Definition \ref{def:frob} one takes $\mathcal{Y} = \OO_K$, the field $L$ to be the Galois closure of $E/K$, the set $C$ to consist of the identity element of $\Gamma$, and $S$ the set of those primes of $K$ which ramify in $L/K$.
\end{example}

A generalisation \cite[Thm.~9.1]{SerNXp} of the prime number theorem to higher-dimensional arithmetic schemes states that
\begin{equation} \label{eqn:prime_number_theorem}
	\#\{ y \in \underline{\mathcal{Y}} : \Norm(y) \leq B\} \sim \Li(B^{n+1}), \quad \mbox{as } B \to \infty,
\end{equation}
where $$\Li(x) = \int_{2}^x \frac{dt}{\log t}.$$
Hence we define the \emph{density} of a subset $F \subset \underline{\mathcal{Y}}$ to be
$$\dens(F) = \lim_{B \to \infty} \frac{\#\{y \in F: \Norm(y) \leq B\}}{\Li(B^{n+1})},$$
if the limit exists. The set of closed points of $\mathcal{Y}$ of degree greater than $1$ satisfies \cite[Lem.~9.3]{SerNXp}
\begin{equation} \label{eqn:large_degree}
	\#\{ y \in \underline{\mathcal{Y}} : \Norm(y) \leq B, \, \deg y > 1\} \ll B^{n + 1/2},
\end{equation}
in particular, it has density zero. We also require the following weak version of the Lang--Weil estimate:
\begin{equation} \label{eqn:LW}
 \mathcal{Y}(\FF_\fp) \ll \Norm(\fp)^n.
\end{equation} A generalisation of the Chebotarev density theorem (see \cite[Thm.~9.11]{SerNXp}) says that if $F \subset \underline{\mathcal{Y}}$ is a frobenian set, then the density $\dens(F)$ exists and
\begin{equation} \label{eqn:Cheb}
	\dens(F) = m(F),
\end{equation}
where $m(F)$ is the mean as defined in Definition \ref{def:frob}.
This implies that a frobenian set has positive density if and only if
it is Zariski dense (see \cite[\S 9.3.2, \S 9.3.3]{SerNXp} for details).

We require a strengthening of \eqref{eqn:Cheb}. The following lemma  is probably known to experts; we give a proof for completeness.

\begin{lemma} \label{lem:frob_zeta}
	Let $n=\dim Y$, let $F \subset \underline{\mathcal{Y}}$ be a frobenian set and let
	$$\zeta_F(s) = \prod_{y \in F}\left( 1 - \frac{1}{\Norm(y)^s}\right)^{-1}, \quad \re s > n+1.$$
	Then $\zeta_F(s)$
	admits a holomorphic continuation to the half-plane $\re s \geq n+1$, apart from possibly at the point $s=n+1$,
	where we have
	$$\zeta_F(s) \sim \frac{c_F}{(s-n-1)^{m(F)}} \quad \mbox{for some } c_F \neq 0, \quad \mbox{as } s \to n+1.$$
\end{lemma}
\begin{proof}
	We recall that the usual zeta function of $\mathcal{Y}$ is defined to be
	$$\zeta_\mathcal{Y}(s) = \prod_{y \in \underline{\mathcal{Y}}}\left( 1 - \frac{1}{\Norm(y)^s}\right)^{-1},$$
	and that this is holomorphic and absolutely convergent on $\re s > n+1$, cf. \cite[\S9.1.7]{Ser65}. That $\zeta_F(s)$ is
	holomorphic on $\re s > n+1$ is thus clear.
	
	Choose $C$, $\Gamma$ and $S$ as in Definition \ref{def:frob}.
	Let $\id_C:\Gamma \to \{0,1\}$ be the indicator function of $C$. As this is invariant under conjugation,
	it may be written as a sum $\id_C = \sum_{\chi} \lambda_\chi \chi$
	over the set $\Irr(\Gamma)$ of irreducible characters of $\Gamma$,
	for some $\lambda_\chi \in \CC$ (see \cite[Prop.~2.30]{FH91}).
	It follows that
	\begin{align}
		\zeta_F(s) &  = \prod_{y \in \underline{\mathcal{Y}} \setminus S}
		\left(1 - \frac{\sum_{\chi} \lambda_{\chi} \chi(\Frob_y)}{\Norm(y)^s}\right)^{-1}
		\prod_{y \in F \cap S}\left(1 - \frac{1}{\Norm(y)^s}\right)^{-1}. \label{eqn:Euler1}
	\end{align}
	As $S$ is not Zariski dense, the second Euler product is holomorphic on $\re s > n$.
	For the first one, we appeal to the theory of Artin $L$-functions on arithmetic schemes
	(see \cite[\S2.2]{Ser65}). For an irreducible character $\chi$ of $\Gamma$, we let 
	$$L_S(\chi,s) = \prod_{y \in \underline{\mathcal{Y}} \setminus S}
	\det\left(1 - \frac{\rho_\chi(\Frob_y)}{\Norm(y)^s}\right)^{-1}, \quad \re(s) > n +1,$$
	be the associated Artin $L$-function,
	where $\rho_\chi$ is the irreducible representation corresponding to $\chi$.
	Using the power series expansion of $\log$, we obtain
	\begin{equation} \label{eqn:Artin_log}
		\log L_S(\chi,s) = \sum_{y \in \underline{\mathcal{Y}} \setminus S} \frac{\chi(\Frob_y)}{\Norm(y)^s} 
		+ g(s),
		\quad \re(s) > n +1,
	\end{equation}
	where $g(s)$ is holomorphic on $\re(s) > n + 1/2$.	
		
	If $\chi$ is not the trivial character $\one$, then $L_S(\chi,s)$ is holomorphic without zeros on 
	the half-plane $\re s \geq n+1$, whereas  $L_S(\one,s)$ admits a meromorphic continuation to $\re s \geq n+1$
	without zeros and a single pole of order $1$ at $s = n+1$.
	When $n=0$, this is a classical result for the usual Artin $L$-functions of number fields.
	For general $n$, details can be found in \cite[\S2.6]{Ser65}
	and \cite[\S2.1]{Fal84}. 
	
	Returning to \eqref{eqn:Euler1}, for $\re s > n +1$ we have
	\begin{align*}
		\log \prod_{y \in \underline{\mathcal{Y}} \setminus S}
		\left(1 - \frac{\sum_{\chi} \lambda_\chi \chi(\Frob_y)}{\Norm(y)^s}\right)^{-1} 
		 & = \sum_{\chi \in \Irr(\Gamma)} \lambda_\chi \sum_{y \in \underline{\mathcal{Y}} \setminus S} 
		 \frac{\chi(\Frob_y)}{\Norm(y)^s} + h(s),
	\end{align*}
	where $h(s)$ is holomorphic on $\re s > n + 1/2$. Comparing with \eqref{eqn:Artin_log}
	and exponentiating, we obtain
	\begin{equation} \label{eqn:analytic_cont}
		\zeta_F(s) = G(s) \prod_{\mathclap{\chi \in \Irr(\Gamma)}} L_S(\chi,s)^{\lambda_\chi} , \quad \re s > n + 1,
	\end{equation}	
	where $G(s)$ is holomorphic and non-zero on $\re s > n + 1/2$. The right-hand side of \eqref{eqn:analytic_cont}
	admits a holomorphic continuation to the half-plane $\re s \geq n + 1$, except for a possible branch point
	singularity at $s = n + 1$ coming from the trivial character $\one$.
	The result then follows from the fact that $\lambda_\one = m(F)$.
\end{proof}

\subsection{$\delta$-invariants} \label{sec:delta}
We now study the $\delta$-invariants given by \eqref{def:delta_D}.  
Let $k$ be a perfect field and let $X$ be a scheme of finite type over $k$. Denote by $I$ the set of geometric irreducible components of multiplicity $1$, i.e. irreducible components of $X \otimes_k \overline{k}$ of multiplicity $1$.

The Galois group $\Gal(\overline{k}/k)$ acts on $I$ in a natural way. One can construct an explicit finite group $\Gamma$ through which this action factors, as follows. Denote by $K_1,\ldots,K_r$ the algebraic closures of $k$ inside the function fields of the irreducible components of $X$ over $k$ of multiplicity $1$. Then $\Gamma$ may be taken to be the Galois group of the Galois closure of the compositum of these finite extensions, and $I$ can be identified with  the $\Gamma$-set corresponding to the \'{e}tale $k$-algebra $\prod_{i = 1}^r K_i$; we refer to \cite[Prop.~10.12, p.~A.V.76]{Bou81} for the usual correspondence between $\Gamma$-sets and \'{e}tale $k$-algebras. If $I \neq \emptyset$, we define
\begin{equation} \label{def:delta}
\delta(X) := \frac{\#\{\gamma \in \Gamma : \gamma \mbox{ acts with a fixed point on } I\}}{\# \Gamma}.
\end{equation}
Note that this definition is independent of the choice of $\Gamma$. If $I = \emptyset$, we instead define $\delta(X) = 0$. The quantity $\delta(X)$ measures ``how non-split'' $X$ is. For many arithmetic applications, schemes with $\delta(X) =1$ will turn out to be as good as split schemes.
We first gather some basic properties of these invariants.

\begin{lemma}\label{lem:delta=1}
	Suppose that $\delta(X) \neq 0$.
	Then we have
	$$\frac{1}{\# \Gamma} \leq \delta(X) \leq 1,$$ with $\Gamma$ as constructed above. If $X$ is split, then $\delta(X) = 1$.
\end{lemma}
\begin{proof}
	The first part is trivial.	The second part follows from the observation that if $X$ is split then
	every element of $\Gamma$ acts with a fixed point on $I$. Indeed, the algebraic closure of $k$
	inside the function field of a geometrically integral component is equal to $k$ itself.
\end{proof}

\begin{remark}
	If $\Gamma$ is cyclic, then clearly $\delta(X) = 1$ implies that $X$ is split.
	If $\Gamma$ is non-cyclic, however, then this implication fails in general.
	Indeed, let $K/k$ be a Galois extension with Galois group $\Gamma$ (if one exists). Then the scheme
	$$X:= \bigsqcup_{k \subsetneq  L \subset K} \Spec L$$
	satisfies $\delta(X) = 1$, but is non-split.
\end{remark}

One cannot expect ``simple'' formulae for $\delta(X)$ in general. However,
we have:

\begin{lemma} \label{lem:Galois}
	If $X$ is irreducible but not geometrically integral, then $\delta(X) < 1$.
	If $X$ is integral and the algebraic closure $K$ of $k$ in $k(X)$ is Galois, then
	$$\delta(X) = \frac{1}{[K:k]}.$$
\end{lemma}
\begin{proof}
	If $\delta(X) = 0$ then the first part is clear. If $\delta(X) \neq 0$ then the result
	follows from a theorem of Jordan \cite[Thm.~4]{Ser03}: for a group $\Gamma$ acting
	\emph{transitively} on a finite set $I$ with  $\#I \geq 2$, there is some $\gamma \in \Gamma$ which
	acts without a fixed point.
	The second statement is clear, as in this case $\Gamma$ also acts \emph{freely} on $I$.
\end{proof}
 
\subsection{$\delta$-invariants in families} \label{sec:delta_families}
Given a morphism of schemes $\pi:X\to Y$ of finite type and a (not necessarily closed) point
 $y \in Y$ with $\kappa(y)$ perfect, define
\begin{equation} \label{def:delta_family}
\delta_y(\pi) := \delta(\pi^{-1}(y)).
\end{equation}
We denote by $\Gamma_y(\pi)$ and $I_y(\pi)$ the associated Galois group and $\Gamma_y(\pi)$-set, as defined in the previous section. If the generic fibre
of $\pi$ is split, then there exists a dense open subset over which each fibre is split. This follows from the constructible
nature of geometric integrality \cite[Thm.~9.7.7]{EGAIV}. In the case where the generic fibre is non-split, we now study the extent to which this
non-splitness spreads out in the arithmetic setting.

\begin{proposition} \label{prop:is_frob}
	Let $\mathcal{X}$ and $\mathcal{Y}$ be schemes which are of finite type and flat over $\ZZ$, with $\mathcal{Y}$ integral,
	and let $\pi:\mathcal{X} \to \mathcal{Y}$ be a dominant morphism. Let $\eta$ denote the generic point of $\mathcal{Y}$.
	Then the set
	\begin{equation} \label{def:frob_set}
		\{y \in \underline{\mathcal{Y}}: \pi^{-1}(y) \mbox{ is non-split}\}
	\end{equation}
	is frobenian with density $1 - \delta_\eta(\pi)$.
\end{proposition}
\begin{proof}
	If $\pi^{-1}(\eta)$ has no irreducible component of multiplicity $1$,
	then the same holds on a dense open subset of $\mathcal{Y}$ (see \cite[Lem.~9.7.2]{EGAIV}).
	Hence the complement of  \eqref{def:frob_set} is not Zariski dense and therefore has density $0$.
	As in this case $\delta_\eta(\pi) = 0$ by definition, the result clearly holds.
	
	Assume now that $\pi^{-1}(\eta)$ has at least one irreducible component of multiplicity $1$.
	Removing multiple components if necessary, 
	we may assume that every irreducible component
	of $\pi^{-1}(\eta)$ has multiplicity $1$.
	As strict closed subsets have density $0$, we may also replace $\mathcal{Y}$ 
	by a dense open subset, if required.
	Hence we may assume that the fibre over every $y \in \mathcal{Y}$ has no multiple components,
	since being geometrically reduced is a constructible property \cite[Thm.~9.7.7]{EGAIV}.
	
Let $\Irr_{\mathcal{X}/\mathcal{Y}}$ be the
	``functor of open irreducible components of $\mathcal{X}/\mathcal{Y}$'',
defined by Romagny in \cite[D\'{e}f.~2.1.1]{Rom11}. Shrinking $\mathcal{Y}$ if necessary,
	from \cite[Lem.~2.1.2]{Rom11} and \cite[Lem.~2.1.3]{Rom11} we see that  $\Irr_{\mathcal{X}/\mathcal{Y}}$ is representable by a
	quasi-compact \'{e}tale algebraic space over $\mathcal{Y}$. 
	The generic fibre of $\Irr_{\mathcal{X}/\mathcal{Y}} \to \mathcal{Y}$ 
	finite: it is the finite \'etale $\kappa(\eta)$-scheme corresponding to the $\Gamma_\eta(\pi)$-set $I_\eta(\pi)$.
	In particular, shrinking $\mathcal{Y}$ if necessary,
	we may assume that 
	$\Irr_{\mathcal{X}/\mathcal{Y}}$ is finite \'{e}tale over $\mathcal{Y}$, hence
	$\Irr_{\mathcal{X}/\mathcal{Y}}$ is a scheme by Knutson's criterion \cite[Cor.~II.6.16]{Knu71}.
	
	Let $y \in \underline{\mathcal{Y}}$. By the definition of $\Irr_{\mathcal{X}/\mathcal{Y}}$,
	the fibre $\pi^{-1}(y)$ is split if and only if the fibre of $\Irr_{\mathcal{X}/\mathcal{Y}}$ over $y$ has a $\kappa(y)$-point.
	By the usual Galois correspondence for finite \'etale schemes, this occurs if and only if $\Frob_y \in \Gamma_\eta(\pi)$ 
	acts with a fixed point on $I_{\eta}(\pi)$. Hence the set \eqref{def:frob_set}
	is frobenian, and by the Chebotarev density theorem \eqref{eqn:Cheb} it has density $1 - \delta_\eta(\pi)$, as required.
\end{proof}

From the above proposition and the Chebotarev density theorem \eqref{eqn:Cheb},
we deduce the following corollary, which will be used in the proof of Theorem \ref{thm:iff}.
\begin{corollary} \label{cor:is_frob}
	Under the assumptions of Proposition \ref{prop:is_frob}, assume furthermore that $\delta_\eta(\pi)=1$.
	Then the set
	$$\{y \in \underline{\mathcal{Y}}: \pi^{-1}(y) \mbox{ is non-split}\}$$
	is not Zariski dense in $\underline{\mathcal{Y}}$.
\end{corollary}

We now give an alternative interpretation of the $\delta$-invariant
for a variety over a number field, suggested to us by Alexei Skorobogatov:

\begin{lemma} 
	Let $X$ be a smooth variety over a number field $k$.
	Then the set $\{\fp \subset \OO_k: X(k_\fp) \neq \emptyset\}$ is frobenian with density $\delta(X)$.
\end{lemma}
\begin{proof} Let $S$ be a sufficiently large finite set of primes of $k$ for which there exists a smooth model $\mathcal{X} \to \Spec \OO_{k,S}$ for $X$. The proof of Proposition \ref{prop:is_frob} shows that by enlarging $S$ if necessary, we may also assume that
	$\Irr_{\mathcal{X}/\OO_{k,S}}$ is representable by a finite \'{e}tale scheme over $\Spec \OO_{k,S}$.

	By Proposition \ref{prop:is_frob},  the set 
	$$\{\fp \subset \OO_{k,S}: \mathcal{X}_{\FF_\fp} \mbox{ is split}\}$$ is frobenian with density $\delta(X)$. If $\mathcal{X}_{\FF_\fp}$ is split and $\fp$ is sufficiently large, the Lang--Weil estimates 
	and Hensel's lemma together imply that $X(k_\fp) \neq \emptyset$. If  $\mathcal{X}_{\FF_\fp}$ is non-split and $\fp \not\in S$, then 
	$\Irr_{\mathcal{X}/\OO_{k,S}}(\FF_\fp)=\Irr_{\mathcal{X}/\OO_{k,S}}(k_\fp) = \emptyset$,
	hence $X_{k_\fp}$ is non-split as well. Lemma \ref{lem:smooth_point} then implies that 
	$X(k_\fp) = \emptyset$, as required.
\end{proof}

We now relate the ``$\delta$-invariants'' to the quantity $\Delta(\pi)$  defined in \eqref{def:Delta}.
Let $\pi:X \to Y$ be a morphism of varieties over a number field $k$ with geometrically integral generic fibre,
with $Y$ integral.
Fix a model for $\pi$ over $\OO_{k,S}$, for a finite set of primes $S$, i.e.~a morphism of flat schemes of finite type $\mathcal{X} \to \mathcal{Y}$ over $\Spec \OO_{k,S}$ (again denoted by $\pi$), such that the induced morphism $\mathcal{X} \otimes_{\OO_{k,S}} k \to \mathcal{Y} \otimes_{\OO_{k,S}} k$ is identified with the original morphism $\pi$. Define
\begin{equation} \label{def:Delta_Y}
\Delta(\pi)
 := \lim_{B \to \infty} \frac{\sum_{\substack{\fp \subset \OO_{k,S} \\ N(\fp) \leq B}} \#\{y_\fp \in \mathcal{Y}(\FF_\fp) :\pi^{-1}(y_\fp) \mbox{ is non-split}\}}
{\sum_{\substack{\fp \subset \OO_{k,S} \\ N(\fp) \leq B}}N(\fp)^{n-1} }, 
\end{equation}
where $n = \dim Y$.
This definition is easily seen to be independent of the choice of $S$ and the corresponding model.
We now show that the limit (\ref{def:Delta_Y}) exists and calculate it; this will prove the first part of Theorem \ref{thm:non-split}.

\begin{proposition} \label{prop:Delta}
	The limit in \eqref{def:Delta_Y} exists and we have
	$$\Delta(\pi) = \sum_{D \in Y^{(1)}} (1 - \delta_D(\pi)).$$
\end{proposition}
\begin{proof}
	Since the generic fibre of $\pi$ is geometrically integral, there exists a strict closed subset 
	$T$ of $Y$ such that the fibre above  each $y \in Y \setminus T$ is split.
	Thus by Lemma \ref{lem:delta=1}, the sum above can be taken over the
    points $D$ in $Y^{(1)} \cap T$.

	Let $\mathcal{T}$ denote the Zariski closure of $T$ in $\mathcal{Y}$. Enlarging $S$ if necessary,
	we find that for all $\fp \notin S$,
	if the fibre over $y \in \mathcal{Y}_{\FF_\fp}$ is non-split, then $y \in \mathcal{T}$.
	Hence
	\begin{equation} \label{eqn:Delta2}
	\Delta(\pi) 
	 = \lim_{B \to \infty} \frac{\sum_{\substack{\fp \subset \OO_{k,S} \\ N(\fp) \leq B}} 
	\#\{y_\fp \in 	\mathcal{T}(\FF_\fp) :\pi^{-1}(y_\fp) \mbox{ is non-split}\}}
	{\sum_{\substack{\fp \subset \OO_{k,S} \\ N(\fp) \leq B}} N(\fp)^{n-1} }.
	\end{equation}
	If $T$ has no irreducible component of codimension $1$ then the limit \eqref{eqn:Delta2} 
	is zero by \eqref{eqn:LW},
	so the proposition trivially holds. 
	Similarly, the intersection of any two irreducible components of $T$ of codimension $1$
	does not contribute to \eqref{eqn:Delta2}. As prime ideals $\fp$ of degree 
	greater than $1$ do not affect the limit \eqref{eqn:Delta2},	we therefore obtain
	\begin{align*}
	\Delta(\pi)
	& = \sum_{D \in Y^{(1)} \cap T} 
	\dens\left(\{y \in \underline{\mathcal{T}_D} :\deg(y)=1, \pi^{-1}(y) \mbox{ is non-split}\} \right),
	\end{align*}
	where $\mathcal{T}_D$ denotes the closure of $D$ in $\mathcal{T}$.
	As the set of closed points of degree larger than $1$ has density zero \eqref{eqn:large_degree},
	the result then follows from applying Proposition \ref{prop:is_frob}
	to each of the $\mathcal{T}_D$ appearing in the above sum.
\end{proof}

We finish by showing that the $\delta$-invariants are birational invariants of the generic fibre, under suitable conditions.
For technical reasons, we work with the class of \emph{almost smooth} morphisms, as defined in \cite[Def.~2.1]{Lou13}.

Let $R$ be a discrete valuation ring with perfect residue field. A separated morphism of finite type 
$\pi: X \to \Spec R$ with smooth generic fibre is said to be \emph{almost smooth} if for any \'{e}tale $R$-algebra $R'$,
each $R'$-point of $X$ lies in the smooth locus of $X$. The canonical example of such a morphism is a dominant morphism $\pi$ with $X$ regular \cite[Prop.~3.1.2]{BLR90}. The extra flexibility given by working with such morphisms sometimes comes in useful for applications, cf.~the proof of Theorem \ref{thm:Serre}.

\begin{lemma} \label{lem:delta_birational}
	Let $\pi_1: X_1 \to \Spec R$ and $\pi_2: X_2 \to \Spec R$ be proper, almost smooth morphisms 
	with $X_1$ and $X_2$ integral. Let $D$ be the closed point of $\Spec R$. If
	there exists a birational map $X_1\dashrightarrow X_2$ compatible with $\pi_1$ and $\pi_2$,
	then we have
	$$\delta_D(\pi_1) = \delta_D(\pi_2).$$
\end{lemma}

\begin{proof}	
	The proof is a minor adaptation of the proof of \cite[Cor.~17.3]{Sko13}. 
	
	Let us first assume that $\delta_D(\pi_1)= 0$, i.e.~that $\pi_1^{-1}(D)$
	does not have an irreducible component  of multiplicity $1$. Then
	$\pi_2^{-1}(D)$ also has no such component by \cite[Lem.~2.2]{Lou13}, hence $\delta_D(\pi_2)=0$
	as required.
	
	Let us now assume that $\delta_D(\pi_1)\neq 0$, hence $\delta_D(\pi_2)\neq 0$.
	Choose a  splitting field, with Galois group $\Gamma_D$, for the irreducible
	components of $\pi_1^{-1}(D)$ and $\pi_2^{-1}(D)$.
	The irreducible components of $\pi_i^{-1}(D)$ of multiplicity $1$ admit
	a natural ordering:  for two such components $c_1$ and $c_2$, we say that
	$c_1 \leq c_2$ if the algebraic closure of $\kappa(D)$ in the function field of $c_1$
	is a subfield of the algebraic closure of $\kappa(D)$ in the function field of $c_2$.

	Let $\min\{I_{D}(\pi_i)\} \subset I_{D}(\pi_i)$ be
	the set of geometric irreducible components of $\pi_i^{-1}(D)$ of multiplicity $1$ which are minimal for this ordering. An element of $\Gamma_{D}$
	acts with a fixed point on $I_{D}(\pi_i)$ if and only if it acts with a fixed point on $\min\{I_{D}(\pi_i)\}$.
	It therefore suffices to show that $\min\{I_{D}(\pi_i)\}$ is a birational invariant,
	i.e.~that
	\begin{equation} \label{eqn:birational}
		\min\{I_{D}(\pi_1)\} \cong \min\{I_{D}(\pi_2)\}
	\end{equation}
	as $\Gamma_D$-sets.
	Under the additional assumption that each $X_i$ is regular, the claim \eqref{eqn:birational}
	is \cite[Cor.~17.3]{Sko13}. The proof of this	goes through if one works with almost 
	smooth morphisms as in \cite[Lem.~2.2]{Lou13}: 
	it suffices to replace the assumption that the total space is regular in the third 
	paragraph of the proof of \cite[Prop.~17.2]{Sko13} by the assumption that the morphism
	is almost smooth. This completes the proof.
\end{proof}

\section{Proof of results} \label{Sec:proofs}
We now prove the main results stated in the introduction.

\subsection{Proof of Theorem \ref{thm:iff}}
The proof of Theorem \ref{thm:iff} is a minor adaptation
of the proof of \cite[Thm.~3.8]{BBL15}. The key new input is the following generalisation of \cite[Cor.~3.7]{BBL15}:

\begin{proposition} \label{prop:split_codim_2}
	Let $\pi:X \to Y$ be a dominant morphism of varieties over a number field $k$ with geometrically integral generic fibre and $Y$ integral. 
	Assume that $\Delta(\pi) = 0$. Then there exist a finite set of primes $S$ of $k$, a model $\mathcal{X} \to \mathcal{Y}$ for $\pi$ over $\OO_{k,S}$  (again denoted by $\pi$) and a closed subset $\mathcal{Z} \subset \mathcal{Y}$ of codimension 
	at least two, such that the induced map
	\[
	(\mathcal{X} \setminus \pi^{-1}(\mathcal{Z})) (\OO_{\fp}) \to 
	(\mathcal{Y} \setminus \mathcal{Z})(\OO_{\fp})
	\]
	is surjective for all primes $\fp \notin S$.
\end{proposition}
\begin{proof}	
	Choose a finite set of primes $S$ of $k$ and a model $\pi: \mathcal{X} \to \mathcal{Y}$ for $\pi$
	over $\OO_{k,S}$.
	As in the proof of Proposition \ref{prop:Delta}, enlarging $S$ if necessary, there is
	a strict closed subset $T \subset Y$ such that for all $\fp \notin S$
	and all points $y_\fp \in \mathcal{Y}_{\FF_\fp} \setminus \mathcal{T}_{\FF_\fp}$, the fibre $\pi^{-1}(y_\fp)$
	is split (here $\mathcal{T}$ denotes the closure  of $T$ in $\mathcal{Y}$).
	
	We claim that, assuming $\Delta(\pi) = 0$ and enlarging $S$ if necessary,
	there is a closed subset $\mathcal{Z} \subset \mathcal{T} \subset \mathcal{Y}$ of 
	codimension at least two in $\mathcal{Y}$,
	such that for all $\fp \notin S$ and all \emph{closed} points 
	$y_\fp \in \underline{\mathcal{Y}}_{\FF_\fp} \setminus \underline{\mathcal{Z}}_{\FF_\fp}$,
	the fibre $\pi^{-1}(y_\fp)$ is split.
	
	The claim is trivial if $\mathcal{T}$ itself has codimension at least two. 
	If not, let $D$ be an irreducible component of $T$ which has codimension $1$ in $Y$ 
	and denote by $\mathcal{T}_D$ the closure of $D$ in $\mathcal{T}$.
	As $\delta_D(\pi) = 1$,
	Corollary \ref{cor:is_frob} yields a dense open subscheme $\mathcal{U} \subset \mathcal{T}_D$
	such that the fibre over all $u \in \underline{\mathcal{U}}$ is split. Applying
	this to all such components $D$ of $T$ proves the claim.
	
	Hence, as in the proof of \cite[Lem.~3.6]{BBL15} and \cite[Cor.~3.7]{BBL15} and enlarging $S$ if necessary,
	we may apply Deligne's estimates \cite[Thm.~1]{Del80} to deduce that the fibre $\pi^{-1}(y_v)$ over every point
	$y_\fp \in \mathcal{Y}(\FF_\fp) \setminus \mathcal{Z}(\FF_\fp)$ contains a smooth $\FF_\fp$-point.
	The conclusion then follows from Hensel's lemma.
\end{proof}

Let $\pi:X \to \PP_k^n$ be as in Theorem \ref{thm:iff}, and $\mathcal{Z} \subset \PP_{\OO_k}^n$ as in 
Proposition \ref{prop:split_codim_2}. For each place $v$ of $k$,
let $\mu_v$ be a choice of Haar measure on $k_v^{n+1}$, 
normalised so that $\mu_\fp(\OO_\fp^{n+1})=1$ for all primes $\fp$ of $k$.
We wish to apply the version of the sieve of Ekedahl given in \cite[Prop.~3.4]{BBL15}, taking
$$\Omega_{v} = \pi(X(k_v)),$$
for any place $v$ of $k$ in the notation of \emph{loc.~cit}.

First consider the affine cone $\Omega_v^\mathrm{aff} \subset k_v^{n+1}$ of $\Omega_v$.
It follows from \cite[Lem.~3.9]{BBL15}
that $\Omega_v^\mathrm{aff}$ is measurable with respect to $\mu_v$ and has boundary of measure zero.
Moreover, our assumption $X(\Adele_k) \neq \emptyset$ implies that 
$\mu_v(\Omega_v^\mathrm{aff}) > 0$ for all $v$ (again see \cite[Lem.~3.9]{BBL15}).
This shows that our $\Omega_v$ satisfy the conditions $\mu_v(\partial \Omega_v^\mathrm{aff}) = 0$
and $\mu_v(\Omega_v^\mathrm{aff}) > 0$ of \cite[Prop.~3.4]{BBL15}.
Next, it follows from Proposition \ref{prop:split_codim_2}
that 
$$\{x \in \PP^n(\OO_{\fp}): x \bmod \fp \not \in \mathcal{Z}(\FF_\fp) \}
\subset \pi(X(k_\fp))$$
for all but finitely many primes $\fp$. Combining this with \cite[Lem.~3.5]{BBL15} we see that the
hypothesis (3.5) of \cite[Prop.~3.4]{BBL15} is also satisfied. We may therefore apply \cite[Prop.~3.4]{BBL15} 
to deduce that
\begin{align*}
\lim_{B \to \infty} \frac{N_{\mathrm{loc}}(\pi,B)}{\#\{x \in \PP^n(k): H(x) \leq B\}}  
= & \prod_{v \mid \infty} \frac{\mu_v(\{\x \in \pi(X(k_v))^\mathrm{aff}: \max_i{|x_i|_v} \leq 1\})}
	{\mu_v(\{\x \in k_v^{n+1}: \max_i{|x_i|_v} \leq 1\})}	\\
& \times	\prod_{\fp} \mu_\fp(\{\x \in \pi(X(k_\fp))^\mathrm{aff} \cap \OO_\fp^{n+1}\}),
\end{align*}
and that this limit is non-zero.
This completes the proof of Theorem \ref{thm:iff}.
\qed

\subsection{Proof of Theorem \ref{thm:non-split_integral}}
We now prove Theorem \ref{thm:non-split_integral}, by applying the large sieve to the criterion from Theorem \ref{thm:sparsity}. We will prove Theorem \ref{thm:non-split} using Theorem \ref{thm:non-split_integral} in the next section. We let $\pi:X \to \mathbb{A}_k^n$ be as in Theorem \ref{thm:non-split_integral} and choose an $\RR$-vector space norm $\|\cdot\|$ on $k_\infty^n$. We choose a sufficiently
large set of primes $S$ of $k$ and a model $\pi: \mathcal{X} \to \mathbb{A}^n_{\OO_{k,S}}$ for $\pi$ over $\OO_{k,S}$;
in the proof we allow ourselves to enlarge $S$ if necessary.

\subsubsection{The large sieve}
Define  $\mu$ multiplicatively on ideals of $\OO_k$ via 
$$\mu(\fp) = -1, \qquad \mu(\fp^m) = 0, \,\, m > 1,$$
for any prime ideal $\fp$.
We will need the following version of the large sieve, cf.~\cite[Lem.~6.2]{BBL15}:

\begin{proposition}\label{prop:large_sieve}
Let $m,n\in \NN$,  let $B\geq 1$ and 
let $\Omega \subset \OO_k^{n}$. For each prime ideal $\fp$,
 assume that there exists  $\omega(\fp) < 1$ such that    
the image of $\Omega$ in $(\OO_k/\fp^m)^{n}$ has at most $(1-\omega(\fp))(\Norm\fp)^{mn}$ elements. Then 
$$
\#\{\x\in \Omega: \|\x\| \leq B\} \ll \frac{B^{n}}{L(B^{1/(2m)})},
\,\, \mbox{where }
L(B)=\sum_{\substack{\fa \subset \OO_k \\ \Norm(\fa) \leq B}} 
|\mu(\fa)|\prod_{\fp\mid \fa} \frac{\omega(\fp)}{1-\omega(\fp)}.
$$
\end{proposition}
\begin{proof}
	When $k = \QQ$, this is given in \cite[\S 6]{Ser90}. The extension to number fields is standard
	and follows a strategy similar to the version given in \cite[\S 12.1]{Ser97b}.
\end{proof}

\subsubsection{The sieving set}
Choose a square-free  polynomial $$f\in \OO_{k,S}[x_1,\ldots,x_n]$$
such that the closed subscheme $T = \{f = 0\}$ of $\mathbb{A}^n_k$ contains the non-smooth locus of $\pi$.
Let $\mathcal{T}$ be the closure of $T$ in $\mathbb{A}^n_{\OO_{k,S}}$.
We now apply Theorem \ref{thm:sparsity} to obtain a criterion which is amenable to the large sieve.
In the next statement, we write $\fp \| f(\x)$ to mean that $\fp \mid f(\x)$ 
(as ideals), but $\fp^2 \nmid f(\x)$.

\begin{proposition}  \label{prop:upper_bound}
	Enlarging $S$ if necessary,
	there exists $g\in \OO_{k,S}[x_1,\ldots,x_n]$
	which is coprime to $f$ such that
	$$N_{\mathrm{loc}}(\pi,B) \leq \# \{ \x \in \OO_k^{n}:\|\x \| \leq B, 
	\x \bmod \fp^2 \notin A(\fp) \text{ for all } \fp \notin S\},$$
	where 
	\begin{equation} \label{def:A(p)}
	A(\fp) = \{\x \in (\OO_{\fp}/\fp^2)^{n} : \fp \| f(\x), \fp \nmid g(\x),
	\pi^{-1}(\x \bmod \fp) \mbox{ is non-split}\}.
	\end{equation}
\end{proposition}
\begin{proof}
	Enlarging $S$ if necessary, we can choose  a closed subset $\mathcal{Z} \subset \mathcal{T}$ 
	of codimension $2$ in $\mathbb{A}^n_{\OO_{k,S}}$
	which satisfies the conditions of Theorem \ref{thm:sparsity}.
	Let us also choose a polynomial $g\in \OO_{k,S}[x_1,\ldots,x_n]$
	which vanishes on $\mathcal{Z}$ and such that $g$ is coprime to $f$, i.e. so that $g=0$ contains no component
	of $T$.

	Theorem \ref{thm:sparsity} then implies that for $\fp \notin S$ and $\x \in \mathbb{A}^n(\OO_{k})$,
	if $\fp \mid f(\x)$, $\fp \nmid g(\x)$ and $\pi^{-1}(\x \bmod \fp) \mbox{ is non-split}$, and if moreover $\x$
	intersects $\mathcal{T}$ transversally above $\fp$, then $\pi^{-1}(\x)$ has no $k_\fp$-point.
	However, if $\fp \| f(\x)$ then a simple local computation shows that $\x$ meets $\mathcal{T}$ transversally above $\fp$, cf.~the argument given in the proof of Proposition 4.3 in \cite[\S6]{BBL15}.
	This completes the proof.
\end{proof}
We now calculate the cardinality of the excluded residues \eqref{def:A(p)}. Let
$$r(\fp) = \#\{t \in \mathcal{T}(\FF_\fp): \pi^{-1}(t) \mbox{ is non-split}\}.$$

\begin{lemma} \label{lem:r}
	We have
	$$|A(\fp)| =  r(\fp) \Norm(\fp)^{n} + O(\Norm(\fp)^{2(n-1)}), \quad \mbox{as } \Norm(\fp) \to \infty.$$
\end{lemma}
\begin{proof}
	We claim that
	\begin{equation} \label{eqn:f^2}
		\#\{\x \in (\OO_{\fp}/\fp^2)^{n} : f(\x) \equiv 0 \bmod \fp^2\} \ll \Norm(\fp)^{2(n-1)}.
	\end{equation}
	Under the additional assumption that $f$ is homogeneous, this is \cite[Lem.~6.3]{BBL15}.
	To deduce \eqref{eqn:f^2}, we apply \cite[Lem.~6.3]{BBL15} to the homogenisation $F$ of $f$ to find
	that 
	\begin{equation} \label{eqn:F^2}
	 \#\{(\x,y) \in (\OO_{\fp}/\fp^2)^{n+1} : F(\x,y) \equiv 0 \bmod \fp^2\} \ll \Norm(\fp)^{2n},
	\end{equation}
	where $F(\x,1) = f(\x)$. Considering the usual diagonal action of the unit group 
	$(\OO_{\fp}/\fp^2)^*$ and using \eqref{eqn:F^2}, we obtain
	\begin{align*}
		& \#\{\x \in (\OO_{\fp}/\fp^2)^{n} : f(\x) \equiv 0 \bmod \fp^2\} \\
		& = \#(\{(\x,y) \in (\OO_{\fp}/\fp^2)^{n+1} : y \in (\OO_{\fp}/\fp^2)^*, F(\x,y) \equiv 0 \bmod \fp^2\}/(\OO_{\fp}/\fp^2)^*) \\
		& \ll \Norm(\fp)^{2n}/ \#(\OO_{\fp}/\fp^2)^* \\
		& \ll \Norm(\fp)^{2(n-1)},
	\end{align*}
	which proves \eqref{eqn:f^2}. Applying \eqref{eqn:LW} and \eqref{eqn:f^2}, we find that
	\begin{align*}		
	|A(\fp)| 
	 = & \#\{\x \in (\OO_{\fp}/\fp^2)^{n} : f(\x) \equiv 0 \bmod \fp, \pi^{-1}(\x \bmod \fp) \mbox{ non-split}\} \\
	 &\  +O(\Norm(\fp)^{2(n-1)}) \\
	 = & \#\{\x \in \FF_\fp^{n} : f(\x) = 0, \pi^{-1}(\x) \mbox{ is non-split}\}  \Norm(\fp)^{n} 
	 + O(\Norm(\fp)^{2(n-1)}) \\
	 = & r(\fp) \Norm(\fp)^{n} + O(\Norm(\fp)^{2(n-1)}). \qedhere
	\end{align*}
\end{proof}

\subsubsection{Application of the large sieve}
We now apply Proposition \ref{prop:large_sieve} with $m=2$ and 
$$\omega(\fp) = \frac{|A(\fp)|}{\Norm(\fp)^{2n}}.$$
The following lemma gives a lower bound for $L(B)$.

\begin{lemma} \label{lem:L(B)} 
	Let $\Delta(\pi)$ be as in \eqref{def:Delta}. Then
	$$\sum_{\substack{\fa \subset \OO_k \\ \Norm(\fa) \leq B}} |\mu(\fa)|
	\prod_{\fp \mid \fa} \frac{\omega(\fp)}{1 - \omega(\fp)} \gg (\log B)^{\Delta(\pi)}.$$
\end{lemma}
\begin{proof}
	As $\omega(\fp) \leq 1$, it suffices to show that 
	\begin{equation} \label{eqn:asym}
	\sum_{\substack{\fa \subset \OO_k \\ \Norm(\fa) \leq B}} |\mu(\fa)|
	\prod_{\fp \mid \fa} \omega(\fp) \sim C(\log B)^{\Delta(\pi)},
	\end{equation}
	for some $C>0$. To do this, consider the associated Dirichlet series
	\begin{equation*} \label{def:Dirichlet}
		\Psi(s)	= \sum_{\fa \subset \OO_k } \frac{|\mu(\fa)|\prod_{\fp \mid \fa} \omega(\fp)}{\Norm(\fp)^s}
		=\prod_\fp\left(1 + \frac{\omega(\fp)}{\Norm(\fp)^s}\right),
		\quad \re s > 0.
	\end{equation*}
	By Lemma \ref{lem:r} we have
	\begin{equation} \label{eqn:Psi_r}
	\Psi(s) =  g(s)\prod_\fp\left(1 + \frac{r(\fp)}{\Norm(\fp)^{s + n}}\right),
	\end{equation}
	where $g(s)$ is holomorphic on $\re s > -1/2$.
	To continue, for each $D \in T^{(0)}$ let $T_D$ (resp.~$\mathcal{T_D}$) 
	denote the closure of $D$ in $T$ (resp.~$\mathcal{T}$).
	Let
	$$F_D = \{t \in \underline{\mathcal{T_D}}: \pi^{-1}(t) \mbox{ is non-split}\},
	\quad r_D(\fp) = \#\{t \in F_D \cap \mathcal{T_D}(\FF_\fp)\}.$$
	Note that $F_D$ is frobenian of density
	$1-\delta_D(\pi)$ by Proposition \ref{prop:is_frob}.
	By \eqref{eqn:LW}, the contribution from the intersection of any $2$ 
	of the  $\mathcal{T_D}$ is negligible. Hence from \eqref{eqn:Psi_r} we
	obtain
	\begin{align*}
	\log \Psi(s)
	& = \sum_\fp \frac{r(\fp)}{\Norm(\fp)^{s + n}}
	+ g_1(s)\\
	& = \sum_{D \in T^{(0)}} \sum_\fp  \frac{r_D(\fp)}{\Norm(\fp)^{s + n}}
	+ g_2(s) \\
	& = \sum_{D \in T^{(0)}} \sum_{\substack{t \in F_D \cap \mathcal{T}_D(\FF_\fp)}}
	\frac{1}{\Norm(t)^{s + n}}
	+ g_2(s),
	\end{align*}
	where each $g_i$ is holomorphic on $\re s > -1/2$.
	Since the contribution of closed points of degree greater than $1$ is negligible \eqref{eqn:large_degree},
	on exponentiating  we find 
	$$\Psi(s) = G(s) \prod_{\mathclap{D \in T^{(0)}}} \zeta_{F_D}(s+n), \quad \re s > 0, $$
	where $G(s)$ is holomorphic and non-zero on $\re s > -1/2$, and $\zeta_{F_D}$ is as in Lemma \ref{lem:frob_zeta}
	(note that $\dim T_D = n-1$).
	Hence from Lemma \ref{lem:frob_zeta} and Proposition \ref{prop:Delta} we find that
	$\Psi(s)$ has a holomorphic continuation to $\re s \geq 0$,
	apart from at $s = 0$ where 
	$$\Psi(s) \sim \frac{c}{s^{\Delta(\pi)}} \quad \mbox{for some } c>0, \quad \mbox{as } s \to 0.$$
	A Tauberian theorem (e.g.~\cite[\S II.7.3, Thm.~8]{Ten95})
	yields \eqref{eqn:asym}, as required.
\end{proof}

Using Proposition \ref{prop:upper_bound} and Lemma \ref{lem:L(B)}, the large sieve gives the upper bound required for
Theorem \ref{thm:non-split_integral}. Recalling Proposition \ref{prop:Delta} completes
the proof.
\qed

\subsection{Proof of Theorem \ref{thm:non-split}}
We prove the result using Theorem \ref{thm:non-split_integral}. Let $\pi: X \to \PP^n_k$ be as in Theorem \ref{thm:non-split}.
Let $X_c$ be the variety obtained as the fibre product
$$
\xymatrix{
X_c \ar[d] \ar[r]^{\pi_c\,\,\,\,\,\,\,\,\,\,\,}&	\mathbb{A}_k^{n+1}\setminus \{0\} \ar[d]^\psi \\
X \ar[r]^{\pi}& \PP^n_k.}
$$
where $\psi$ is the obvious map.
The morphism $\pi_c$ is still proper, but in order to apply Theorem \ref{thm:non-split_integral} the base needs to be
$\mathbb{A}_k^{n+1}$. We therefore choose an open immersion $X_c \subset \widetilde{X}$ where $\widetilde{X}$ is a smooth, integral $k$-variety equipped with a proper map 
$\widetilde{\pi}: \widetilde{X} \to \mathbb{A}_k^{n+1}$ such that $\widetilde{\pi}|_{X_c} = \pi_c$; this is possible 
by Nagata's compactification theorem and Hironaka's theorem. The diagram
$$
\xymatrix{
X_c \ar@{^{(}->}[d] \ar[r]^{\pi_c\,\,\,\,\,\,\,\,\,\,}&	\mathbb{A}_k^{n+1}\setminus \{0\} \ar@{^{(}->}[d] \\
\widetilde{X} \ar[r]^{\widetilde{\pi}}& \mathbb{A}_k^{n+1}.}
$$
commutes. For any $P \in \mathbb{A}_k^{n+1} \setminus \{0\}$, we have an isomorphism
\begin{equation} \label{eqn:fibres_agree}
	\widetilde{\pi}^{-1}(P) \cong \pi^{-1}(\psi(P)) \otimes_{\kappa(\psi(P))} \kappa(P).
\end{equation}
It follows that for all non-zero $\x \in  \mathbb{A}_k^{n+1}(k)$ we have
\begin{equation} \label{eqn:pi=pi}
    \psi(\x) \in \pi(X(\Adele_k)) \iff \x \in \widetilde{\pi}(\widetilde{X}(\Adele_k)).
\end{equation}
We next compare the $\Delta$-invariants.

\begin{lemma} \label{lem:Delta=Delta}
	In the above notation we have
	\begin{equation} 
	\Delta(\pi) = \Delta(\widetilde{\pi}).
\end{equation}
\end{lemma}
\begin{proof}
	There are two types of codimension one points $D$ in $\mathbb{A}_k^{n+1}$:
	\begin{enumerate}
		\item[(I)] $\psi(D)$ is the generic point $\eta$ of $\PP_k^{n}$,
		\item[(II)]	$\psi(D)$ has codimension $1$ in $\PP_k^{n}$.
	\end{enumerate}
	
	Since the generic fibre of $\pi$ is geometrically integral, using \eqref{eqn:fibres_agree} we see that $\widetilde{\pi}^{-1}(D)$ is also geometrically integral for points $D$ of type (I). Hence
	$$\delta_D(\widetilde{\pi}) = 1$$ for such points. 
	For a point $D$ of type (II), we have 
	$$\kappa(D) = \kappa(\psi(D))(t)$$
	where $t$ is purely transcendental over $\kappa(D)$, and hence (again by \eqref{eqn:fibres_agree})
	$$\delta_D(\widetilde{\pi}) = \delta_{\psi(D)}(\pi).$$
	The result then follows from Proposition \ref{prop:Delta}, since
	$\psi$ induces a bijection between $(\PP^n_k)^{(1)}$ and the set of codimension
	one points of $\mathbb{A}_k^{n+1}$ of type (II).
\end{proof}
Consider now the counting function $N_{\mathrm{loc}}(\widetilde{\pi},B)$ from \eqref{def:Nloc_affine},
for some choice of norm $\|\cdot \|$.
By the proposition in \cite[\S13.4]{Ser97b}, there exists a constant $C > 0$
such that every point $x \in \PP^n(k)$ has a choice of coordinates $\x \in \OO_k^{n+1}$ 
with
$$||\x|| \leq CH(x).$$
It follows from this and \eqref{eqn:pi=pi} that
\begin{equation*} 
	N_{\mathrm{loc}}(\pi,B) \leq N_{\mathrm{loc}}(\widetilde{\pi},CB).
\end{equation*}
Applying Theorem \ref{thm:non-split_integral} and recalling Lemma \ref{lem:Delta=Delta} completes the proof. \qed

\subsection{Proof of Theorem \ref{thm:single_fibre}}
This follows immediately from
Theorem \ref{thm:non-split} and Lemma \ref{lem:Galois}. \qed

\section{Examples and applications} \label{Sec:examples}

The aim of this section is to illustrate our results with a few examples and applications. 

\subsection{Preliminaries}
Let us gather some preliminary definitions and results.

\begin{definition} \label{def:goodmodels} 
Let $k$ be a field. If $\pi: X \to \PP^n_k$ is a dominant morphism of $k$-varieties with smooth, geometrically integral generic fibre, we say that $\psi: Y \to \PP^n_k$ is a \emph{good compactification} of $\pi$ if $\psi$ is proper, $Y$ is smooth over $k$, and the generic fibre of $\psi$ contains the generic fibre of $\pi$ as an open subvariety. 
\end{definition}

Good compactifications always exist if $k$ is a field of characteristic zero: this follows from Nagata's compactification theorem and resolution of singularities. 

\begin{lemma} \label{lem:independence} Let $k$ be a number field. Let $\pi: X \to \PP^n_k$ and $\psi: Y \to \PP^n_k$ be as in Definition \ref{def:goodmodels}. Then the sets $\pi(X(\Adele_k)) \cap \PP^n(k)$ and $\psi(Y(\Adele_k)) \cap \PP^n(k)$ differ by a non-Zariski-dense set of rational points. In particular, we have 
\begin{equation} \label{eqn:equal}
 N_{\mathrm{loc}}(\pi,B) = N_{\mathrm{loc}}(\psi,B) + O(B^{n+1/2 + \varepsilon}) 
\end{equation} 
 for any $\varepsilon > 0$.
\end{lemma}

\begin{proof} There exists an open subset $U \subseteq \PP^n_k$ such that the induced morphism $\psi^{-1}(U) \to U$ is smooth, and $\pi^{-1}(U)$  is an open subset of $\psi^{-1}(U)$. If $x \in U(k)$ is a rational point and $v$ is a place of $k$, then $\psi^{-1}(x)$ has a $k_v$-point if and only if $\pi^{-1}(x)$ has a $k_v$-point; indeed, on a smooth, connected $k_v$-variety, the set of rational points is either empty or Zariski dense. This proves the first statement. The claim \eqref{eqn:equal} then follows from \cite[Thm.~13.1.3]{Ser97b}.  \end{proof}

This lemma is important for applications: often, one wants to study  a family of varieties given by an explicit set of affine equations. 
However, the morphism $\pi$ in Theorem \ref{thm:non-split} is assumed to be proper; this assumption is necessary, cf. Example \ref{Ex:proper}. Hence one needs a good compactification of $\pi$ to apply the result. Lemma \ref{lem:delta_birational} and Lemma \ref{lem:independence} imply that the bounds provided by Theorem \ref{thm:non-split} are independent of the choice of such a compactification. It is not always necessary to construct a good compactification explicitly in order to apply Theorem \ref{thm:non-split}; see e.g. \S \ref{sec:tori} for a case where this can be avoided.

Before listing examples and applications, we will present a non-example, showing that the smoothness hypothesis in Theorems \ref{thm:single_fibre} and \ref{thm:non-split} is crucial:

\begin{example}[Necessity of assumptions, I] Consider the conic bundle $$\pi: X \to \mathbb{A}^1_\QQ = \mathrm{Spec}\,\QQ[t]$$ given by $$x^2 + y^2 = t^2z^2.$$ Even though the fibre over $t = 0$ is irreducible but not geometrically integral, there is a section, so the conclusion of Theorem \ref{thm:single_fibre} does not hold here; the issue is that total space of this family is singular at $t= 0$ and $(x,y,z) = (0,0,1)$.
\end{example}

\subsection{Torsors under multinorm tori} \label{sec:tori}
Our first application of Theorem \ref{thm:non-split} concerns families of torsors under multinorm tori. 
Let $k$ be a number field and let $E/k$ be a finite \'{e}tale $k$-algebra.
Consider the family $$\pi: X \to \mathbb{A}^1_k = \mathrm{Spec}\,k[t]$$ given by
\begin{equation} \label{multinorm} 
\Norm_{E/k} (\x) = t.
\end{equation} 
Away from $t = 0$, the fibres are torsors under the multinorm torus $R^{1}_{E/k}\GG_m$. In order to apply Theorem \ref{thm:non-split} and to calculate the $\delta$-invariants, we need a good compactification as in Definition \ref{def:goodmodels}. However, it is not necessary to construct such a compactification explicitly; the following lemma suffices:

\begin{lemma}
\label{lem:nonsplit} Let $K$ be a field of characteristic zero. Let $E = \prod_{i = 1}^r K_i$ be an \'etale $K$-algebra, where the $K_i/K$ are finite field extensions of degree $n_i$. Let $X$ be the $K(\!(t)\!)$-variety given by the equation $$\Norm_{E/K}(\mathbf{x}) = t^m, \ \ \text{for some}\ m \in \mathbb{Z}.$$
Let $\mathcal{X}$ be a regular scheme equipped with a proper morphism $\psi: \mathcal{X} \to \mathrm{Spec}\,K[\![t]\!]$ whose generic fibre is smooth, geometrically integral, and contains $X$ as an open subscheme. Then the special fibre $\mathcal{X}_s$ is split if and only if $$\gcd(n_1,\ldots,n_r) \mid m.$$ 
\end{lemma}

\begin{proof} 
We prove the result using \cite[Lem.~2.2]{Sko96}.

If $\gcd(n_1,\ldots,n_r) \mid m$, then $\psi$ has a section. Indeed, pick $s_1,\ldots,s_r \in \ZZ$ with $n_1s_1 + \cdots + n_r s_r = m$. Then a section is given by  $t \mapsto (t^{s_1},\ldots,t^{s_r})$. As $\mathcal{X}$ is regular and $\pi$ is proper, the special fibre is split by \cite[Lem.~2.2(b)]{Sko96}.

Conversely, assume that $\mathcal{X}_s$ is split. Then 
\cite[Lem.~2.2(a)]{Sko96} implies that there is an unramified extension of discrete valuation rings $K[\![t]\!] \subseteq A$ such that $A$ is complete, $K$ is algebraically closed in the residue field of $A$, and $\mathcal{X}(A) \neq \emptyset$.
Hence the generic fibre of $\mathcal{X}$ has a $K(A)$-point.

Recall that $K(A)$ is a so-called \emph{large field} by \cite[\S1.A.2]{Pop14}: this means that every irreducible curve over $K(A)$ with a smooth rational point has infinitely many rational points. The generic fibre of $K(A)$ is smooth and irreducible; hence, as $K(A)$ is large, the rational points are Zariski dense by \cite[Prop.~2.6]{Pop14}. In particular the affine patch of the generic fibre given by 
$$\Norm_{E/K}(\mathbf{x}) = \prod_{i=1}^r \Norm_{K_i/K} (\mathbf{x}_i) = t^m$$ has a $K(A)$-point $(z_1,\ldots,z_r)$.  Taking valuations yields $$\sum_{i = 1}^r n_i v_A(z_i) = v_A(t^m) = m.$$ Hence $\gcd(n_1,\ldots,n_r) \mid m$, as required.
\end{proof}

We now apply Theorem \ref{thm:non-split} to the family \eqref{multinorm}. Define $\delta_{E/k}$ to be the density
\begin{equation}
	\delta_{E/k} = \dens\left(\left\{\fp \in \Spec \OO_k \left|
	\,\,\,\,\,\, \gcd_{ \mathclap{k_\fp \subset L_\fp \subset E_\fp}} \,\, [L_\fp : k_\fp]   = 1 \right. \right\}\right).
\end{equation} Here $E_\fp = E \otimes_k k_\fp$, and the greatest common divisor is taken over all subfields $L_\fp$ of $E_\fp$ containing $k_\fp$.
That this density exists follows from Chebotarev's density theorem \eqref{eqn:Cheb}, as will be clear from the proof of the following result:

\begin{theorem} \label{thm:norm}
	Consider the family $\pi$ given by (\ref{multinorm}). Then
	\begin{equation}
		N_{\mathrm{loc}}(\pi,B) \ll \frac{B^2}{(\log B)^{2(1-\delta_{E/k})}}. \label{eqn:BN}
	\end{equation}
\end{theorem}
\begin{proof}
Let $\psi: Y \to \PP^1_k$ be a good compactification of $\pi$. By Lemma \ref{lem:nonsplit}, any non-split fibre must  lie over $t = 0$ or $t = \infty$. To deduce \eqref{eqn:BN} from Theorem~\ref{thm:non-split}, it suffices to prove the equalities
\begin{equation} \label{eqn:delta=delta}
	\delta_0(\psi) = \delta_\infty(\psi) = \delta_{E/k}.
\end{equation}
We check this for $t=0$, the case $t = \infty$ being similar. 

Let $\fp$ be a prime of $k$ which is unramified in $E$ and let $\psi_\fp$ denote the base change of $\psi$ to $k_\fp$. Let $\Gamma_0(\psi)$ and $I_0(\psi)$ be as in the definition \eqref{def:delta_D} of $\delta_0(\psi)$. The fibre of $\psi_\fp$ over $0$ is split if and only if $\Frob_\fp$ acts with a fixed point on $I_0(\psi)$. However, by Lemma \ref{lem:nonsplit}, the fibre of $\psi_\fp$ over $0$ is split if and only if the greatest common divisor of the degrees $[L_\fp : k_\fp]$ is equal to $1$, where $L_\fp$ runs over all subfields of $E_\fp$ which contain $k_\fp$. The claim \eqref{eqn:delta=delta} then follows from the Chebotarev density theorem.
\end{proof}

\begin{remark}
When $k = \QQ$ and $E$ is a number field, the family \eqref{multinorm} was studied by Browning--Newton. In \cite[Thm.~1.3]{BN15},
they obtained asymptotic formulae for the counting functions \eqref{def:rational} and \eqref{def:local}. They proved that
\begin{equation} \label{BrowningNewton} 
N_{\mathrm{loc}}(\pi,B) \sim c_{E/\QQ} \frac{B^2}{(\log B)^{2(1-\delta_{E/\QQ})}} 
\end{equation}
for some $c_{E/\QQ} > 0$. Theorem \ref{thm:norm} hence  gives a sharp upper bound in this case.
\end{remark}

\begin{remark}
	For any given \'etale $k$-algebra $E$, it is possible to calculate $\delta_{E/k}$ using the Chebotarev density theorem,
	though one should not expect simple expressions in general when $E/k$ is not Galois.
	One nice case is when $E/k$ is a field extension of prime degree $p$ whose Galois closure
	has Galois group $S_p$; a simple exercise shows that $\delta_{E/k} = 1 - 1/p$ in this case.
\end{remark}
We now give an example of a family of torsors under a torus which illustrates that the properness assumption in Theorem \ref{thm:non-split} is crucial:

\begin{example}[Necessity of assumptions, II]  \label{Ex:proper}
Let $K$ (resp.~$L$) be a quadratic (resp.~cubic) field extension of a number field $k$ and consider the family \eqref{multinorm} for $E=K \times L$. It has the equation $$\Norm_{K/k}(\mathbf{x}) \Norm_{L/k}(\mathbf{y}) = t.$$ The fibre of $\pi$ over $0$ is non-split and $\delta_0(\pi) < 1$, hence $\Delta(\pi) > 0$. The map $t \mapsto (t^{-1},t)$ defines a rational section, and every  fibre has a rational point. This does not contradict Theorem \ref{thm:non-split}: $\pi$ is not proper, and Lemma \ref{lem:nonsplit} implies that the fibres of a good compactification of $\pi$ are all split.\end{example}

The next example illustrates that one cannot replace the assumptions in Theorem \ref{thm:single_fibre} by the weaker condition
that there is a non-split fibre.

\begin{example}[Necessity of assumptions, III] \label{Ex:CT}
Let $a,b,ab \notin k^{* 2}$ and consider \eqref{multinorm} with $E = k(\sqrt{a}) \times k(\sqrt{b}) \times k(\sqrt{ab})$. 
It has the equation
$$\Norm_{k(\sqrt{a})/k}(\mathbf{x}) \Norm_{k(\sqrt{b})/k}(\mathbf{y}) \Norm_{k(\sqrt{ab})/k}(\mathbf{z}) = t.$$
This family was studied by Colliot-Th\'{e}l\`{e}ne in \cite{CT14}.
Let $\psi:Y \to \PP_k^1$ be a good compactification in the sense of Definition \ref{def:goodmodels}. The fibres over $t = 0$ and $t = \infty$ are non-split by Lemma \ref{lem:nonsplit}. However, the fibre over every rational point is everywhere locally solvable by \cite[Prop.~5.1]{CT14}. 
Hence the existence of a non-split fibre is not enough to  deduce the conclusion of Theorem \ref{thm:single_fibre}. 

This example is compatible with Theorem \ref{thm:iff} however. Indeed, we have $\delta_0(\pi) = \delta_\infty(\pi) = 1$, as every element of the Galois group of the polynomial
$$f(x)= (x^2 - a)(x^2 - b)(x^2 - ab) \quad \in k[x]$$ acts with a fixed point on the roots of $f$.
\end{example}

\subsection{Specialisations of Brauer group elements}
Theorem \ref{thm:non-split} allows us to recover both \cite[Thm.~2]{Ser90} and Theorem 8.2 of \cite[Ch.~II, Appendix]{Ser97a}. It also yields the expected generalisations of these results, already hinted at by Serre in \S8 of \cite[Ch.~II, Appendix]{Ser97a}, to finite collections of Brauer group elements and to arbitrary number fields $k$.

Let $U \subset \PP^n_k$ be a dense open subset and let $\br \subset \Br U$ be a finite subset. Define
$$N(\br,B) = \#\{x \in U(k) : H(x) \leq B, b(x) = 0 \in \Br k \ \text{for all}\ b \in \br \}.$$
Our generalisation of the above-mentioned results is the following.

\begin{theorem} \label{thm:Serre}
	We have
	$$N(\br,B) \ll \frac{B^{n+1}}{(\log B)^{\Delta(\br)}}, \quad \mbox{where} \quad
	\Delta(\br) = \sum_{D \in (\PP^n_k)^{(1)}} 
	\left(1 - \frac{1}{|\partial_D(\langle \br \rangle)|}\right).$$
	Here $\partial_D$ denotes the residue map at $D$ and $\langle \br \rangle$ the subgroup generated by $\br$.
\end{theorem}
\begin{proof}
	Most of the tools needed to translate this to a problem about Severi--Brauer
	schemes can be found in \cite{Lou13}.
	For $b\in \br$, let $V_b \to U$ denote the corresponding Severi--Brauer scheme. Let
	$V = \prod_{b \in \br} V_b \to U$
	be their fibre product, taken over $U$. Fix a good compactification $\pi:X \to \PP^n_k$ of $V \to U$.
	Standard properties of Severi--Brauer schemes
	\cite[\S2.3]{Lou13} give
	$$N(\br,B) = \#\{x \in U(k) : H(x) \leq B, x \in \pi(X(k))\}.$$
	Hence Theorem \ref{thm:non-split} yields
	$$N(\br,B) \ll \frac{B^{n+1}}{(\log B)^{\Delta(\pi)}}.$$
	To complete the proof, it suffices to show that for all $D \in (\PP^n_k)^{(1)}$ we have
	\begin{equation} \label{eqn:delta=res}
		\delta_D(\pi) = \frac{1}{|\partial_D(\langle \br \rangle)|}.
	\end{equation}
	We shall prove this using Lemma \ref{lem:delta_birational} and the special models
	constructed in \cite[Lem.~2.3]{Lou13}. Let $R$ be the local ring at $D$.
	By \cite[Lem.~2.3]{Lou13} there exists an integral proper almost smooth scheme
	$$\psi:\mathcal{V} \to \Spec R$$
	whose generic fibre is isomorphic to the generic fibre of $\pi$, with the property that
	the algebraic closure of $\kappa(D)$ in the function field of each irreducible component of 
	$\psi^{-1}(D)$ is isomorphic to the compositum $K$ of the cyclic field extensions determined by the residues
	$\res_D(b) \in \mathrm{H}^1(\kappa(D), \QQ/\ZZ)$ for $b \in \br$. In particular, the set
	of irreducible components $I_D(\psi)$ is isomorphic to the $\Gamma_D(\psi)$-set corresponding
	to the finite \'{e}tale scheme $(\Spec K)^n$, for some $n \in \NN$.
	As $K/\kappa(D)$ is Galois of degree $|\partial_D(\langle \br \rangle)|$,
	a argument similar to the proof of Lemma \ref{lem:Galois} yields
	$$\delta_D(\psi) = \frac{1}{|\partial_D(\langle \br \rangle)|}.$$
	However, by Lemma \ref{lem:delta_birational} we have $\delta_D(\pi) = \delta_D(\psi)$,
	whence \eqref{eqn:delta=res}.
\end{proof}

As a special case, let $\pi: X \to \PP^n_k$ be a conic bundle, with $X$ non-singular. 
Then Theorem \ref{thm:non-split} shows that
\begin{equation*}
 N_{\mathrm{loc}}(\pi,B) \ll \frac{B^{n+1}}{(\log B)^{\Delta(\pi)}}, \quad  
\Delta(\pi) = \frac{1}{2}\#\{D \in (\PP^n_k)^{(1)}: \pi^{-1}(D) \text{ is non-split} \}.
\end{equation*}

\subsection{Fermat curves} \label{sec:Fermat}
Consider the family of Fermat curves of degree $d \geq 2$
$$X: \quad a_0x_0^d + a_1x_1^d + a_2x_2^d = 0 \quad \subset \PP_k^2 \times \PP_k^2$$
 over a number field $k$, equipped with the natural projection $\pi:X \to \PP_k^2$ to $(a_0 : a_1 : a_2)$. 
 This family was studied over $\QQ$ in \cite{BD09}. When $k=\QQ$, \cite[Thm.~1]{BD09} gives
\begin{equation} \label{eqn:BD}
	N_{\mathrm{loc}}(\pi,B) \ll \frac{B^3}{(\log B)^{3\psi(d)}}, \quad \psi(d) = \frac{1}{\varphi(d)}\left(1 - \frac{1}{d}\right)
\end{equation}
where $\varphi$ is Euler's totient function.
Theorem \ref{thm:non-split} gives the following result:

\begin{theorem} \label{thm:FermatQ}
	For $d\geq 2$ and $k = \QQ$ we have
	\begin{equation} \label{eqn:Fermat}
		N_{\mathrm{loc}}(\pi,B) \ll \frac{B^3}{(\log B)^{3(1-\delta(d))}},
	\end{equation}
	where $\delta$ is the multiplicative function given by
	$$\delta(p^m) =1- \frac{p^{2m} - 1}{p^{2m-1}(p^2-1)},$$ for any prime $p$ and for any positive integer $m$.
\end{theorem}
\begin{proof}
	The non-split fibres over codimension $1$ points lie over the generic points of the three divisors $D_i$ given by $a_i = 0$, for $0 \leq i \leq 2$. On choosing an isomorphism $\kappa(D_i) \cong \QQ(a)$ for some purely transcendental
	element $a$, we see that
	the finite \'{e}tale $\kappa(D_i)$-scheme corresponding to each $\Gamma_{D_i}(\pi)$-set 
	$I_{D_i}(\pi)$ is given by the zero locus of the polynomial
	$$f_d(x) = x^d + a \quad \in \QQ(a)[x].$$
	The Galois group $\Gamma_d$ of the Galois closure of $f_d$ is
	\begin{equation} \label{eqn:Gamma}
		\Gamma_d \cong \ZZ/d\ZZ \rtimes (\ZZ/d\ZZ)^\times.
	\end{equation}
	As $X$ is non-singular, an application of Theorem \ref{thm:non-split} now yields 
	an upper bound of the shape \eqref{eqn:Fermat} with
	\begin{equation} \label{def:F(d)}
		\delta(d) = \frac{F(d)}{d\varphi(d)},
	\end{equation}
	where $F(d)$ is the number of elements of $\Gamma_d$ which act with a fixed point on the roots of $f_d(x)$.
	With respect to the isomorphism \eqref{eqn:Gamma}, the group $\Gamma_d$ acts on the roots
	as the group of affine linear transformations on $\ZZ/d\ZZ$, i.e.~$(s,t)\cdot n = s + tn.$
	Hence we need to count the number of pairs $(s,t) \in \ZZ/d\ZZ \times (\ZZ/d\ZZ)^\times$ 
	for which the equation $(1-t)n = s$ has a solution for some $n \in \ZZ/d\ZZ$; this is the case precisely when
	$\gcd(t-1,d) \mid s$. Therefore we obtain
	\begin{equation} \label{eqn:F(d)}
		F(d) = \#\{(s,t) \in \ZZ/d\ZZ \times (\ZZ/d\ZZ)^\times:   s \equiv 0 \bmod \gcd(t-1,d)\}.
	\end{equation}
	The Chinese remainder theorem now implies that
	$F$ is a multiplicative function of $d$. To determine 
	$F(p^m)$ when $p$ is prime and $m$ is a positive integer, we use \eqref{eqn:F(d)} to
	find that
	\begin{align*}
		F(p^m) &= \sum_{i = 0}^m \#\{s \in \ZZ/p^m\ZZ: v_p(s) \geq i\} \cdot 
		\#\{t \in (\ZZ/p^m\ZZ)^\times: v_p(t - 1) = i\} \\ 
		&= p^{2m-1}(p - 2) + 1 + (p-1)\sum_{i = 1}^{m - 1} p^{2(m-i)-1}  \\
		&= p^{2m-1}(p-2) + 1 + \frac{p(p^{2(m - 1)} - 1)}{p+1}		\\
		&= p^{2m-1}(p-1) - \frac{p^{2m} - 1}{p+1}.
	\end{align*}
	Combining this with $\eqref{def:F(d)}$ completes the proof.
\end{proof}
	
Comparing Theorem \ref{thm:FermatQ} with \eqref{eqn:BD}, 
we find $1 - \delta(p) = \psi(p) = 1/p$. However our result is stronger,
since we have $1 - \delta(p^m) > \psi(p^m)$ for all $m \geq 2$. 

Note that Conjecture \ref{question} is known in this case when $d = 2$ and $k=\QQ$, by independent work of Hooley \cite{Hoo93} and Guo \cite{Guo95}. Other (non-sharp) lower bounds have been recently obtained by Dietmann and Marmon \cite[Lem.~5]{DM15}.
The upper bounds we obtain change as one varies the number field.
For example, if $\mu_d \subset k$, then Theorem \ref{thm:non-split} gives
the following bound, which is strictly stronger in general than the bound one obtains over $\QQ$:
\begin{theorem}
	If $d \geq 2$ and $\mu_d \subset k$, then
	$$N_{\mathrm{loc}}(\pi,B) \ll \frac{B^3}{(\log B)^{3(1-1/d)}}.$$
\end{theorem}
\begin{proof}
	As $\mu_d \subset k$, the polynomial 
	$x^d + a$
	over $k(a)$ defines a  Galois extension of degree $d$.
	On applying Lemma \ref{lem:Galois}	and Theorem \ref{thm:non-split},
	the result follows in a manner similar to the proof of Theorem \ref{thm:FermatQ}.
\end{proof}

\subsection{Genus $1$ fibrations: a question of Graber--Harris--Mazur--Starr}
Let $\pi:X \to \PP^1_k$ be a pencil of genus $1$ curves over a number field $k$.
Some questions on the behaviour of the counting function \eqref{def:rational} 
in such cases were raised by Graber--Harris--Mazur--Starr in \cite{GHMS04}. 
They studied the counting function
$$N^*(\pi, B) = \#\{x \in \PP^1(k): x \notin \pi(X(k)), H(x) \leq B\},$$
which is the complement of \eqref{def:rational}. Question $3$ of \cite{GHMS04} is the following.
\begin{question}[Graber--Harris--Mazur--Starr] \label{q:GHMS}
	If $\pi$ does not admit a section, then does there exist $e>0$ such that
	$$N^*(\pi,B) \gg B^e \quad ?$$
\end{question}
If $\Delta(\pi) > 0$ then Theorem \ref{thm:non-split} answers this question in the affirmative. In fact, we obtain the stronger 
statement that
$$N^*(\pi,B) \sim \#\{x \in \PP^1(k): H(x) \leq B\}, \quad \mbox{as } B \to \infty.$$
In particular, Question \ref{q:GHMS} has a positive answer in such cases with $e = 2$.
We now give some explicit examples to which our results apply.

\subsubsection{Families of quadratic twists}
Let us first consider the example studied in \cite[\S 2]{GHMS04}. We briefly recall the construction of genus $1$
fibrations used in \cite{GHMS04}.
Let $U \subset \PP^1_k$ be a dense open subset and let $\mathcal{E}$ be an elliptic scheme over $U$. For $n\in \NN$, Kummer theory
yields an exact sequence
\begin{equation}\label{seq:kummer} \nonumber
 0 \to \mathcal{E}(U)/n\mathcal{E}(U) \to \mathrm{H}^1(U, \mathcal{E}[n]) \to \mathrm{H}^1(U,\mathcal{E})[n] \to 0.
\end{equation}
Any cocycle with class $\alpha \in \mathrm{H}^1(U, \mathcal{E}[n])$ gives rise to an $\mathcal{E}$-torsor $\pi^\alpha: \mathcal{E}^\alpha \to U$,
which admits no section if $\alpha$ does not lie in the image of  $\mathcal{E}(U) \to \mathrm{H}^1(U, \mathcal{E}[n])$.

We now apply this construction with $n=2$. Let 
$$E: y^2 = f(x)$$
be the Weierstrass form of an elliptic curve over $k$.
Consider the family of quadratic twists
\begin{equation} \label{def:quadratic_twists}
	\mathcal{E}: ty^2 = f(x)
\end{equation}
over $U=\mathbb{G}_{\mathrm{m},k} = \mathrm{Spec}\,k[t,\frac1t]$. Let $\alpha \in \mathrm{H}^1(k, E[2])$ be non-zero. 
Note that we have $\mathcal{E}[2] \cong E[2] \times_k U$ as group schemes over $U$. In particular, $\alpha$ 
naturally gives rise to an element of $\mathrm{H}^1(U, \mathcal{E}[2])$.

When $k=\QQ$ and $f$ is irreducible, Graber--Harris--Mazur--Starr \cite[\S 2]{GHMS04} 
used deep results on modular forms due to Ono--Skinner \cite{OS98} and Kolyvagin \cite{Kol90}, to deduce
that 
\begin{equation} \label{eqn:GHMS}
	N^*(\pi^\alpha,B) \gg \frac{B}{\log B}.
\end{equation}
In particular, one may take any $e< 1$ in Question \ref{q:GHMS} in this case. Our sieve methods allow us to bypass these deep modularity results, while simultaneously improving upon \eqref{eqn:GHMS} and obtaining results valid over any number field, for possibly reducible $f$.

\begin{theorem} \label{thm:quadratic_twists}
	Let $\mathcal{E}/U$ be the quadratic twist family of $E$ as in \eqref{def:quadratic_twists}. Let
	$\alpha \in \mathrm{H}^1(k, E[2])$ be non-zero with associated genus $1$ fibration $\pi^\alpha: \mathcal{E}^\alpha \to U$.
	Then
	\begin{equation} \label{eqn:quadratic_twists_bound}
		N_{\mathrm{loc}}(\pi^\alpha,B) \ll \frac{B^2}{(\log B)^{1/2}}.
	\end{equation}
	In particular, Question \ref{q:GHMS} has a positive answer in this case with $e = 2$.
\end{theorem}
\begin{proof}
	We will construct the minimal proper regular model of $\mathcal{E}^\alpha$ and show that
	there are non-split fibres over $0$ and $\infty$.
	
	Let $\mathcal{E}^{\mathrm{N}}$ be the N\'{e}ron model of $\mathcal{E}$ over $\PP^1_k$.
	By Tate's algorithm, the fibres over $0$ and $\infty$ have Kodaira type $I_0^*$.
	In particular, the geometric fibres over $0$ and $\infty$ are isomorphic to the group scheme
	$(\ZZ/2\ZZ)^2 \times \Ga$: the component group is $(\ZZ/2\ZZ)^2$ by the table in \cite[p.~365]{Sil94},
	and the component group exact sequence splits as $\mathrm{char}(k) = 0$ \cite[Cor.~1.5]{LL01}.
	Let $\mathcal{E}^{\mathrm{N}}[2]$ be the $2$-torsion group scheme of $\mathcal{E}^{\mathrm{N}}$ over $\PP^1_k$.
	From the above, one sees that $\mathcal{E}^{\mathrm{N}}[2]$ is finite \'{e}tale over $\PP^1_k$,
	hence is the base-change of $E[2]$ to $\PP^1_k$. We identify $\alpha$ with its image in 
	$\mathrm{H}^1(\PP^1_k, \mathcal{E}^{\mathrm{N}}[2])$.
	
	Next let $\mathcal{E}^{\mathrm{N}}_c$ be the minimal proper regular model of $\mathcal{E}^{\mathrm{N}}$.
	This admits an action of $\mathcal{E}^{\mathrm{N}}[2]$. In particular, 
	we may twist $\mathcal{E}^{\mathrm{N}}_c$ by a cocycle with class $\alpha$ to obtain a genus $1$ fibration
	$$\pi_c^{\mathrm{N}, \alpha}: \mathcal{E}_c^{\mathrm{N}, \alpha} \to \PP^1_k,$$
	which is the minimal proper regular model of $\mathcal{E}^\alpha$.
	
	We now consider the singular fibres of $\pi_c^{\mathrm{N}, \alpha}$. We focus on the fibre
	over $0$, as the analysis at $\infty$ is analogous.	
	The group scheme $E[2]$ acts transitively	on the geometric irreducible components of $\mathcal{E}^{\mathrm{N}}$ above $0$. 
	We find that the set $I_0(\pi_c^{\mathrm{N}, \alpha})$
	is an $E[2]$-torsor with class $\alpha$. As this is not the trivial torsor, it has no rational point.
	This shows that the fibre of $\pi_c^{\mathrm{N}, \alpha}$ above $0$ is non-split.
	
	To deduce \eqref{eqn:quadratic_twists_bound} we give a lower bound for $\Delta(\pi_c^{\mathrm{N}, \alpha})$.
	Consider an $E[2]$-torsor with class $\alpha$. This is some finite \'{e}tale scheme of degree $4$ over $k$,
	with Galois action determined by some subgroup of $S_4$ that acts without a fixed point.
	An inspection of such subgroups shows that $\delta_0(\pi_c^{\mathrm{N}, \alpha}) \leq 3/4$. 
	Equality holds for $\Spec K_1 \sqcup \Spec K_2$ where $K_1,K_2$ are different quadratic extensions
	 of $k$, for example. 
	Similarly $\delta_\infty(\pi_c^{\mathrm{N}, \alpha}) \leq 3/4$,
	hence Theorem \ref{thm:non-split} gives the result.
\end{proof}

\subsubsection{Quartic surfaces}
Any line on a smooth quartic surface gives rise to a genus $1$ fibration.
We examine the behaviour obtained for a special family of surfaces considered by Rams and Sch\"{u}tt
\cite[Lem.~4.5]{RS15}. 
Let $\mathcal{Z} \subset \PP_k^3 \times \mathbb{A}_k^8$ be the family of all smooth quartics
of the form
$$x_1^3 x_3 + x_2^3 x_4 + x_1x_2q(x_3,x_4) + g(x_3,x_4) = 0 \quad \subset \PP_k^3,$$ where $\deg q = 2$ and $\deg g = 4$. We view this over a suitable dense open subset $\mathcal{Z} \to U \subset \mathbb{A}_k^8$, given by the choice of coefficients for $q$ and $g$ for which one obtains a smooth surface. Each such surface contains the line $x_3 = x_4 = 0$. For $u \in U$, we denote the corresponding surface over $\kappa(u)$ by $\mathcal{Z}_u$.  We first study the generic surface $\mathcal{Z}_\eta$ over the function field $\kappa(\eta)$ of $\mathbb{A}_k^8$.

\begin{lemma} \label{lem:quartic}
	Let $\pi_\eta: \mathcal{Z}_\eta \to \PP_{\kappa(\eta)}^1$ be the generic surface in  $\mathcal{Z}$ and let $P$
	be the closed point of $\PP_{\kappa(\eta)}^1$ given by $q^3 + 27 x_3 x_4 g = 0$. The fibre over $P$ is 
	irreducible and becomes split over a pure cubic extension of $\kappa(P)$. 
	The fibre above every other point of $\PP_{\kappa(\eta)}^1$ is geometrically integral.
\end{lemma}
\begin{proof}
	As explained at the end of \cite[\S4]{RS15},
	there are exactly twelve singular fibres over the algebraic closure of $\kappa(\eta)$.
	Six have type $I_1$, hence are
	irreducible. The other six have type $I_3$ and lie above $P$. This
	proves the second statement.
	To prove the first statement, let 
	$P_0: x_4 = \alpha x_3$ be a point of $P \otimes_{\kappa(\eta)} L$, where $L=\kappa(P)$, and write
	$q(x_3,\alpha x_3) = \beta x_3^2$ for some $\beta \in L$.
	The fibre over $P_0$ takes the form
	$$\alpha x_1^3 + \alpha^2 x_2^3 + \alpha\beta x_1x_2x_3 - \beta^3 x_3^3/27 = 0.$$
	This is the norm form of $L(\alpha^{1/3})/L$ with respect to the
	basis $(\alpha^{1/3},\alpha^{2/3}, -\beta/3)$. In particular, one easily sees
	that this scheme is irreducible and contains a line over $L(\alpha^{1/3})$.
	This proves the result.	
\end{proof}

We now come to the application of Theorem \ref{thm:non-split}.

\begin{theorem}
	Let $k$ be a number field and let $\mathcal{Z}$ be as above. There exists a thin subset $T \subset U(k)$ with the following
	property. Let $u \in U(k) \setminus T$ and let $\pi_u:\mathcal{Z}_u \to \PP_k^1$ be the induced
	genus $1$ fibration. Then
	$$N_{\mathrm{loc}}(\pi_u,B) \ll \frac{B^2}{(\log B)^{\Delta(\pi_u)}}, 
	\quad \mbox{where} \quad\Delta(\pi_u) = 
	\begin{cases}
		1/3, & \mu_3 \not \subset k, \\
		2/3, & \mu_3 \subset k.	
	\end{cases}
	$$
\end{theorem}
\begin{proof}
	Consider the map $\mathcal{Z} \to U \times \PP_k^1$ induced by the natural morphism $\mathcal{Z} \to U$ and the universal
	genus $1$ fibration $\mathcal{Z} \to \PP_k^1$. Let $P \in \PP_{\kappa(\eta)}^1$
	be as in Lemma~\ref{lem:quartic} and let $F$ be the closure of the pull-back of the point $\eta \times P$ to $\mathcal{Z}$.
	This is a closed subscheme of $\mathcal{Z}$ such that $F \cap \mathcal{Z}_\eta$ is exactly the non-split fibre from Lemma \ref{lem:quartic}.
	We see that there exists a dense open subset $V \subset U$ such that $F \cap \mathcal{Z}_u$ is a singular fibre
	of $\pi_u$ for all $u \in V$. An application of Hilbert's irreducibility theorem \cite[Prop.~3.3.1]{Ser08}
	and Lemma \ref{lem:quartic} moreover imply that, outside of some thin subset, the scheme
	$F \cap \mathcal{Z}_u$
	is a non-split fibre of $\pi_u$  which is irreducible and split by a pure cubic extension.
	The result then follows from Theorem \ref{thm:non-split},
	together with a calculation similar to the one in \S \ref{sec:Fermat}.
\end{proof}

\end{document}